\documentclass[10pt]{amsart}

\usepackage{amsmath,amssymb,latexsym,hyperref,url,graphicx}
\usepackage[all]{xy}
\usepackage[usenames,dvipsnames]{color}

\newtheorem{theorem}{Theorem}[section]
\newtheorem{corollary}[theorem]{Corollary}
\newtheorem{lemma}[theorem]{Lemma}

\newtheorem*{conjecture*}{Conjecture}
\newtheorem*{theorem*}{Theorem}

\theoremstyle{definition}
\newtheorem*{definition}{Definition}
\newtheorem*{example}{Example}

\numberwithin{equation}{section}

\newcommand{\Aut}{\operatorname*{Aut}}
\newcommand{\Inn}{\operatorname*{Inn}}

\newcommand{\gray}[1]{\textcolor{Gray}{#1}}

\begin{document}

\title{Classification of automorphic conjugacy classes in the free group on two generators}

\author{Bobbe Cooper}
\address{
	School of Mathematics \\
	University of Minnesota \\
	Minneapolis, MN 55455, USA
}

\author{Eric Rowland}
\address{
	Laboratoire de combinatoire et d'informatique math\'ematique \\
	Universit\'e du Qu\'ebec \`a Montr\'eal \\
	Montr\'eal, QC H2X 3Y7, Canada
}
\curraddr{
	Universit\'e de Li\`ege \\
	D\'epartement de Math\'ematiques \\
	Grande Traverse 12 (B37) \\
	4000 Li\`ege, Belgique
}

\begin{abstract}
We associate a finite directed graph with each equivalence class of words in $F_2$ under $\Aut F_2$, and we completely classify these graphs, giving a structural classification of the automorphic conjugacy classes of $F_2$.
This classification refines work of Khan and proves a conjecture of Myasnikov and Shpilrain on the number of minimal words in an automorphic conjugacy class whose minimal words have length $n$, which in turn implies a sharp upper bound on the running time of Whitehead's algorithm for determining whether two words in $F_2$ are automorphic conjugates.
\end{abstract}

\maketitle
\markboth{BOBBE COOPER AND ERIC ROWLAND}{CLASSIFICATION OF AUTOMORPHIC CONJUGACY CLASSES IN $F_2$}

\section{Introduction}\label{Introduction}

We begin with a few standard definitions.
Let $F_2 = \langle a, b \rangle$ be the free group on two generators $a$ and $b$.
The length of $w \in F_2$ is denoted by $|w|$.
A word $w \in F_2$ is \emph{minimal} if $|\phi(w)| \geq |w|$ for all $\phi \in \Aut F_2$.

Two elements $w$ and $v$ in $F_2$ are \emph{automorphic conjugates} if there is an automorphism $\phi \in \Aut F_2$ such that $\phi(w) = v$.
We write $w \sim v$ if $w$ and $v$ are automorphic conjugates.
Equivalence classes under $\sim$, which we refer to as \emph{automorphic conjugacy classes}, are the main object of study in this paper.

An automorphic conjugacy class $W$ supports a natural graph structure in which the vertices are the words in $W$ and a directed edge is drawn from $w$ to $v$ for each automorphism $\phi$ such that $\phi(w) = v$.
Here we will be interested in the subgraph consisting of minimal words, say of length $n$, and in particular we will define (in Section~\ref{The graph}) a quotient $\Gamma(W)$ of this subgraph obtained by dividing by $n$ inner automorphisms and $8$ permutations.

The size of $\Gamma(W)$ has implications for the running time of a standard algorithm for determining whether two words in $F_2$ are automorphic conjugates.
To bound the time complexity of this algorithm, Myasnikov and Shpilrain~\cite{Myasnikov--Shpilrain} studied the number of minimal words in an automorphic conjugacy class $W$.
They showed that if $w \in F_2$ is a minimal word of length $n$, then the number of minimal words in its automorphic conjugacy class is bounded above by a polynomial in $n$.
Further, they conjectured that $8n^2 - 40n$ gives a sharp bound for $n \geq 9$.
In terms of $\Gamma(W)$, where we have divided by $8 n$ automorphisms, this is equivalent to the statement that $|V(\Gamma(W))| \leq n - 5$ for $n \geq 9$.
Khan~\cite{Khan} showed that this conjectured bound holds for sufficiently large classes.
His approach was to identify a number of subgraphs that $\Gamma(W)$ avoids and use these subgraphs to bound the number of vertices.

\begin{theorem*}[Khan]
If $W$ is an automorphic conjugacy class of size $|V(\Gamma(W))| \geq 4373$ whose minimal words have length $n \geq 10$, then $|V(\Gamma(W))| \leq n - 5$.
\end{theorem*}

In this paper we take a direct approach to analyzing the structure of $\Gamma(W)$.
We are able to recast Khan's results with shorter proofs and additional information sufficient to prove the conjecture of Myasnikov and Shpilrain.

\begin{theorem}\label{automorphic conjugacy class size}
If $W$ is an automorphic conjugacy class whose minimal words have length $n \geq 9$, then $|V(\Gamma(W))| \leq n - 5$.
\end{theorem}

Myasnikov and Shpilrain~\cite{Myasnikov--Shpilrain} perceived the possibility of a sharp polynomial bound as quite surprising.
We show in this paper that the structure of automorphic conjugacy classes is quite restricted, perhaps much more so than previously suspected, which accounts for a simple bound.

Our work builds on that of a previous paper~\cite{Cooper--Rowland} in which we identified certain words in $F_2$ as \emph{root words}.
We define these words below, following Theorem~\ref{minimality}.
The property of being a root word is respected by automorphic conjugacy (Theorem~\ref{root classes} below), so each automorphic conjugacy class $W$ can be said to either be a \emph{root class} or a \emph{non-root class}.
For graphs of sufficiently large automorphic conjugacy classes, Khan~\cite{Khan} also identified a dichotomy --- either the number of vertices is bounded by some absolute constant or the graph has at most $n-5$ vertices and simple edge structure.
We show in this paper that the former correspond to root classes and the latter to non-root classes.

Both Khan's approach and ours are founded on a theorem of Whitehead~\cite{Whitehead 1936a, Whitehead 1936b} which provides a finite set of generators for $\Aut F_2$.
Before recalling this theorem we introduce a bit of notation.
Let $L_2 = \{a, b, a^{-1}, b^{-1}\}$.
For $x \in L_2$, denote $\overline{x} = x^{-1}$.
We identify each element $w \in F_2$ with its word on the alphabet $L_2$ in which no pair of adjacent letters are inverses of each other.

A \emph{Type~I automorphism} or a \emph{permutation} is an automorphism which permutes $L_2$.
There are $8$ permutations.

Type~II automorphisms are defined as follows.
Let $x \in L_2$ and $A \subset L_2 \setminus \{x, \overline{x}\}$.
Define a map $\phi : L_2 \to F_2$ by
\[
	\phi(y) = \overline{x}^{\beta(\overline{y} \in A)} \, y \, x^{\beta(y \in A)},
\]
where $\beta(\textsf{true}) = 1$ and $\beta(\textsf{false}) = 0$.
Since $\phi(y)^{-1} = \phi(\overline{y})$ for all $y \in L_2$, this map extends to an automorphism.
We write $\phi = (A,x)$ and call $\phi$ a \emph{Type~II automorphism}.
For example, the automorphism $\phi = (\{a\},b)$ maps $a \mapsto ab$ and $\overline{a} \mapsto \overline{b}\overline{a}$ and leaves $b,\overline{b}$ fixed.
This notation for Type~II automorphisms was introduced by Higgins and Lyndon~\cite{Higgins--Lyndon}; see also the standard book of Lyndon and Schupp~\cite[page~31]{Lyndon--Schupp}.

\begin{theorem*}[Whitehead]
If $w, v \in F_2$ such that $w \sim v$ and $v$ is minimal, then there exists a sequence $\phi_1, \phi_2, \dots, \phi_m$ of Type~I and Type~II automorphisms such that
\begin{itemize}

\item
$\phi_m \cdots \phi_2 \phi_1(w) = v$ and

\item
for $0 \leq k \leq m - 1$, $|\phi_{k+1} \phi_k \cdots \phi_2 \phi_1(w)| \leq |\phi_k \cdots \phi_2 \phi_1(w)|$, with strict inequality unless $\phi_k \cdots \phi_2 \phi_1(w)$ is minimal.

\end{itemize}
\end{theorem*}

To determine whether a word $w$ is minimal, by Whitehead's theorem it suffices to apply each Type~II automorphism to $w$.
Then $w$ is minimal if and only if $|\phi(w)| \geq |w|$ for each Type~II automorphism $\phi$.

In fact we do not need to check all Type~II automorphisms to determine minimality.
For example, $(\{\}, x)$ is the identity automorphism, so we may require that no automorphism $\phi_i$ in Whitehead's theorem is $(\{\}, x)$.

Additionally, notice that $(\{y,\overline{y}\},x)$ is an inner automorphism, since it conjugates $y$ by $x$ and also (trivially) conjugates $x$ by $x$.
We view inner automorphisms as ``cosmetic'' automorphisms, and we will usually dispense with them by dividing $\Aut F_2$ by its normal subgroup $\Inn F_2$.
For clarity, however, our notation will indicate when we have omitted an inner automorphism.
We write $w \equiv v$ if $\phi(w) = v$ for some inner automorphism $\phi$.
Equivalence classes under $\equiv$ are called \emph{cyclic words}.
Let $C_2$ be the set of words $w = x_1 \cdots x_n \in F_2$ such that $x_n \neq \overline{x_1}$.
Words in $C_2$ are representatives of cyclic words.
For the remainder of the paper, all words are elements of $C_2$.
Since $F_2 \setminus C_2$ consists entirely of words which are not minimal, we do not lose any structural information regarding minimal words by moving from $F_2$ to $C_2$.

Since an inner automorphism does not decrease the length of any word in $C_2$, by Whitehead's theorem we need not consider them when determining the minimality of a word in $C_2$.
Therefore the primary automorphisms of interest are automorphisms $\phi = (A, x)$ where $|A| = 1$.
We call such an automorphism a \emph{one-letter automorphism}.
For $y \notin \{x, \overline{x}\}$, the one-letter automorphism $(\{y\},x)$ maps $x \mapsto x$, $\overline{x} \mapsto \overline{x}$, $y \mapsto yx$, and $\overline{y} \mapsto \overline{x}\overline{y}$.
The inverse of $\phi = (\{y\}, x)$ is the one-letter automorphism $\phi^{-1} = (\{y\}, \overline{x})$.

One-letter automorphisms do not commute with permutations in general, but we have the following identity, which we will use a number of times.

\begin{lemma}\label{permutation}
Let $y \notin \{x, \overline{x}\}$, let $\phi = (\{y\}, x)$ be a one-letter automorphism, and let $\pi \in \Aut F_2$ be a permutation.
Then $\pi \phi = (\{\pi(y)\}, \pi(x)) \pi$.
\end{lemma}

\begin{proof}
One checks that both sides map $x \mapsto \pi(x)$ and $y \mapsto \pi(y) \pi(x)$.
\end{proof}

We mention that a consequence of Lemma~\ref{permutation} is that one can pull any permutations in the product $\phi_m \cdots \phi_2 \phi_1$ to the left.
Therefore in Whitehead's theorem one may assume that $\phi_1, \phi_2, \dots, \phi_{m-1}$ are Type~II automorphisms and that $\phi_m$ is a permutation.

There are $8$ one-letter automorphisms; they are given by $(\{y\}, x)$ as $x$ and $y$ run over $L_2$ subject to $y \notin \{x, \overline{x}\}$.
Each one-letter automorphism $(\{y\}, x)$ can be written as the product
\begin{equation}\label{inner automorphism}
	(\{y\},x) = (\{y,\overline{y}\},x) (\{\overline{y}\},\overline{x})
\end{equation}
of an inner automorphism and another one-letter automorphism.
That is, we have $(\{y\},x)(w) \equiv (\{\overline{y}\},\overline{x})(w)$ for all $w \in C_2$.
Therefore, there are only four distinct one-letter automorphisms modulo $\Inn F_2$.
The four \emph{principal} automorphisms are $(\{a\},b)$, $(\{a\},\overline{b})$, $(\{b\},a)$, and $(\{b\},\overline{a})$; they are distinct modulo $\Inn F_2$.
We have shown the following corollary of Whitehead's theorem.

\begin{corollary}\label{decrease the length}
Let $w \in C_2$.
Then $w$ is minimal if and only if none of the principal automorphisms decrease the length of $w$.
\end{corollary}

\begin{example}
Let $w = aa$.
Since the lengths of $(\{a\}, b)(w) = abab$, $(\{a\}, \overline{b})(w) = a\overline{b}a\overline{b}$, $(\{b\}, a)(w) = aa$, and $(\{b\}, \overline{a})(w) = aa$ are at least $2$, $w$ is minimal.
\end{example}

By counting two-letter subwords of $w$ we can determine whether the length of $(\{y\}, x)(w)$ is greater than, less than, or equal to $|w|$.
Hence the minimality of $w$ can be expressed in terms of these subword counts; this is the content of Theorem~\ref{minimality} below.
Our notation for counting subwords is as follows.
If $w = x_1 \cdots x_n$ and $u$ are nonempty words in $C_2$ such that $k = |u| \leq |w| = n$, let $(u)_w$ denote the total number of (possibly overlapping) occurrences of the (contiguous) subwords $u$ and $u^{-1}$ in $x_1 \cdots x_n x_1 \cdots x_{k-1}$.
If $|u| > |w|$, let $(u)_w = 0$.
Essentially we are considering $w$ to be a cyclic word; if $w \equiv w'$ then $(u)_w = (u)_{w'}$.

\begin{example}
Let $w = aa\overline{bb}\overline{a}ba\overline{b}a$; the length-$2$ subword counts are $(aa)_w = 2$, $(bb)_w = 1$, $(ab)_w = 1 = (ba)_w$, and $(a\overline{b})_w = 2 = (\overline{b}a)_w$.
\end{example}

One can show that, in general, $(xy)_w = (yx)_w$ for $w \in C_2$ and $x,y \in L_2$.

In the remainder of this section we give some facts from our previous paper~\cite{Cooper--Rowland} that we will use.
We include a proof of the first lemma to indicate the flavor of the proofs.

\begin{lemma}\label{cyclic word}
Let $w \in C_2$, and let $\phi = (\{y\},x)$ with $y \notin \{x, \overline{x}\}$.
Then
\begin{align*}
	(yy)_{\phi(w)} &= (y\overline{x}y)_w, \\
	(xx)_{\phi(w)} &= (yxy)_w + (yxx)_w + (xxy)_w + (xxx)_w.
\end{align*}
\end{lemma}

\begin{proof}
The only way that $yy$ can occur in $\phi(w)$ is as the image of $y \overline{x} y$ in $w$.
Similarly, $\overline{y} \overline{y}$ occurs in $\phi(w)$ only where $\overline{y} x \overline{y}$ occurs in $w$; this yields the first equality.
The second equality follows from the observation that $xx$ is introduced in $\phi(w)$ where $yxy$ and $yxx$ occur in $w$, and $xx$ in $w$ is preserved under $\phi$ except when followed by $\overline{y}$; similarly for its inverse $\overline{x} \overline{x}$.
\end{proof}

An automorphism $\phi \in \Aut F_2$ is \emph{level} on $w \in C_2$ if $|w| = |v|$ for some $v \in C_2$ such that $v \equiv \phi(w)$.
In other words, $\phi$ is level on $w$ if the lengths of $w$ and $\phi(w)$ as cyclic words are equal.
For example, $(\{b\}, \overline{a})$ is level on $aba\overline{b}$ but is not level on $abab$.

The following lemma is a rephrasing of the statement that a one-letter automorphism is level on $w$ precisely when the number of (cyclic) letter cancellations it causes is equal to the number of additions.
(We must exclude words of length $1$; since cyclically consecutive $a$s in $w = a$ are not actually distinct, there is an addition under $(\{a\}, b)$ that is not captured by counting occurrences of $aa$.)

\begin{lemma}\label{level automorphism}
Let $w \in C_2$ such that $|w| \geq 2$, and let $y \notin \{x, \overline{x}\}$.
Then the automorphism $(\{y\},x)$ is level on $w$ if and only if $(y\overline{x})_w = (yx)_w + (yy)_w$.
\end{lemma}

The next theorem follows easily from Corollary~\ref{decrease the length} and Lemma~\ref{level automorphism}.

\begin{theorem}\label{minimality}
A word $w \in C_2$ is minimal if and only if
\[
	|(ab)_w - (a\overline{b})_w| \leq \min((aa)_w, (bb)_w).
\]
\end{theorem}

Root words are words satisfying the boundary case of this inequality.

\begin{definition}
A word $w \in C_2$ is a \emph{root word} if
\[
	|(ab)_w - (a\overline{b})_w| = (aa)_w = (bb)_w.
\]
\end{definition}

This definition is different than, but equivalent to, the definition used in our previous paper~\cite[Theorem~7]{Cooper--Rowland}.

Examples of root words include $ab\overline{a}\overline{b}$, $aabb$, and $aba\overline{b}$; these words belong to classes~4.2 and 4.3 in Appendix~\ref{Table of automorphic conjugacy classes}, which lists representatives of all classes containing a word of length $n \leq 9$.

\begin{theorem}\label{divisibility by four}
If $w$ is a root word, then $|w|$ is divisible by $4$.
\end{theorem}

An automorphic conjugacy class $W$ is a \emph{root class} if it contains a root word and a \emph{non-root class} if it does not.
Theorem~\ref{root classes} states that all minimal words in a root class are root words.

\begin{theorem}\label{root classes}
If $w$ is a root word, $w \sim v$, and $|w| = |v|$, then $v$ is a root word.
\end{theorem}

A word $w \in C_2$ is \emph{alternating} if $(aa)_w = 0 = (bb)_w$.
For example, $aba\overline{b}$ and $abab$ are alternating.

\begin{theorem}\label{alternating minimal word}
Let $w \in C_2$.
The following are equivalent.
\begin{itemize}
\item
$w$ is an alternating minimal word.
\item
$w$ is an alternating root word.
\item
The four principal one-letter automorphisms are level on $w$.
\end{itemize}
\end{theorem}

The outline of the paper is as follows.
The following section contains the definition of the graph $\Gamma(W)$ and the main theorems of the paper.
These theorems are proved in Sections~\ref{Non-root classes} and \ref{Root classes}.
We conclude in Section~\ref{Enumeration} with conjectures on the number of automorphic conjugacy classes whose minimal words have length $n$.

\section{The graph $\Gamma(W)$}\label{The graph}

In this section we define $\Gamma(W)$, a directed graph associated with an automorphic conjugacy class $W$.
We then state Theorems~\ref{non-root class}--\ref{alternating root class}, which classify these graphs.

The basic idea is to consider a graph where the vertices are minimal words in $W$ and an edge from $w$ to $v$ represents a one-letter automorphism that maps $w$ to $v$.
Note that there are finitely many minimal words in $W$, since there are finitely many words of length $n$.
Therefore the vertex set is finite.
To reduce the number of vertices, we only select distinct minimal words up to ``cosmetic'' similarity.
Namely, if two minimal words are mapped to each other by an inner automorphism and a permutation, then we consider them to be representatives of the same vertex.

More formally, let $J$ be the subgroup of automorphisms of $F_2$ generated by inner automorphisms and permutations.
Write $w \sim_J v$ if $\phi(w) = v$ for some $\phi \in J$.
In particular, if $w \equiv v$ then $w \sim_J v$.
Define $[w]$ to be the equivalence class of $w$ under $\sim_J$, and let the vertices of $\Gamma(W)$ be the equivalence classes of minimal words in $W$ under $\sim_J$.
Note that the vertices in the graphs considered by Khan~\cite{Khan} are equivalence classes modulo inner automorphisms only; hence his graph for an automorphic conjugacy class $W$ has up to $8$ times as many vertices as $\Gamma(W)$ (fewer if there are symmetries in a word).

We now describe the edges of $\Gamma(W)$.
Since $J$ is not a normal subgroup of $\Aut F_2$, we cannot define $\phi([w])$ to be $[\phi(w)]$, because the map $u \mapsto [\phi(u)]$ is not invariant on the minimal words in $[w]$.

\begin{example}
Consider $w = aa$ and $v = bb \in [w]$.
Let $\phi = (\{b\}, a)$.
We have $\phi(w) = w = aa$ and $\phi(v) = baba$, and it is clear that $[aa] \neq [baba]$.
\end{example} 

Instead, if $\phi$ is a one-letter automorphism, let $[\phi]$ be the equivalence class of $\phi$ modulo $\Inn F_2$.
Let $w, v \in C_2$ be minimal words such that $w \sim v$.
We say that $[w]$ is \emph{connected} to $[v]$ by $[\phi]$ if $\phi(w) \in [v]$.
We draw one directed edge in $\Gamma(W)$ from $[w]$ to $[v]$ for each equivalence class $[\phi]$ of one-letter automorphisms such that $[w]$ is connected to $[v]$ by $[\phi]$.

To show that $\Gamma(W)$ is well-defined, we must show that the number of edges from $[w]$ to $[v]$ does not depend on the representatives.
First we show that the property of two vertices being connected does not depend on the 
representatives.
Indeed, suppose that $[w]$ is connected to $[v]$ by $[\phi]$, and let $w' \in [w]$ and $v' \in [v]$.
Then $w' \equiv \pi(w)$ for some permutation $\pi$; letting $\phi' = \pi \phi \pi^{-1}$ gives $\phi'(w') \equiv \pi \phi(w) \in [v] = [v']$.
By Lemma~\ref{permutation}, $\phi'$ is a one-letter automorphism, so $[w']$ is connected to $[v']$ by $[\phi']$.
Note that in general $[\phi'] \neq [\phi]$.
However, the map $\phi \mapsto \pi \phi \pi^{-1}$ is a bijection on the set of one-letter automorphisms.
Moreover, one-letter automorphisms which are equivalent modulo $\Inn F_2$ have images under this map that are equivalent modulo $\Inn F_2$; this can be seen from Lemma~\ref{permutation}.
Therefore the number of edges from $[w]$ to $[v]$ is independent of the representatives chosen.
Hence the graph $\Gamma(W)$ is well-defined.

By Whitehead's theorem, $\Gamma(W)$ is connected.
We see that, by definition, the outdegree of each vertex in $\Gamma(W)$ is at most $4$.
Note that $\Gamma(W)$ can have loops and multiple edges.

\begin{example}
Consider the automorphic conjugacy class $W$ containing the minimal word $aabb$.
This class is class~4.3 in Appendix~\ref{Table of automorphic conjugacy classes}.
The images of $aabb$ under the principal one-letter automorphisms are
\begin{align*}
	(\{a\}, b)(aabb) &= ababbb \\
	(\{a\}, \overline{b})(aabb) &= a\overline{b}ab \\
	(\{b\}, a)(aabb) &= aababa \\
	(\{b\}, \overline{a})(aabb) &\equiv ab\overline{a}b.
\end{align*}
The first and third images are not minimal, so they are not represented in $\Gamma(W)$.
The second and fourth images are elements of $[aba\overline{b}]$, which is distinct from the vertex $[aabb]$.
So let us compute the images of $aba\overline{b}$ under the principal automorphisms:
\begin{align*}
	(\{a\}, b)(aba\overline{b}) &= abba \\
	(\{a\}, \overline{b})(aba\overline{b}) &= aa\overline{bb} \\
	(\{b\}, a)(aba\overline{b}) &= aba\overline{b}			\qquad \text{(a loop)} \\
	(\{b\}, \overline{a})(aba\overline{b}) &= aba\overline{b}	\qquad \text{(a loop)}.
\end{align*}
The first two images are elements of $[aabb]$, so $|V(\Gamma(W))| = 2$ and $\Gamma(W)$ is
\medskip	
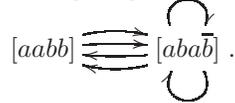
\[
	\xymatrix{
		[aabb] \ar@/^.6pc/[r] \ar@/^.2pc/[r] & \ar@/^.6pc/[l] \ar@/^.2pc/[l] [aba\overline{b}]\ar@(ul, ur)[] \ar@(dr, dl)[]
	}.
\bigskip	
\]
\end{example}

The words listed in Appendix~\ref{Table of automorphic conjugacy classes} for each automorphic conjugacy class are representatives of the vertices of $\Gamma(W)$.
They are the minimal words in $W$ that appear first lexicographically (with the order $a < b < \overline{a} < \overline{b}$ on $L_2$) among their images under inner automorphisms and permutations.
From the listed representatives, one can compute $\Gamma(W)$ by drawing an edge from $[w]$ to $[v]$ for each principal automorphism $\phi$ such that $\phi(w) \sim_J v$.

If there is an edge in $\Gamma(W)$ from $[w]$ to $[v]$ then there is an edge from $[v]$ to $[w]$, since if $\phi(w) = v$ then $\phi^{-1}(v) = w$.
Therefore we say that $[w]$ and $[v]$ are \emph{neighbors} if there is an edge from $[w]$ to $[v]$ (and from $[v]$ to $[w]$) without distinguishing ``out-neighbors'' from ``in-neighbors''.

Note, however, that the number of edges from $[w]$ to $[v]$ is not necessarily equal to the number of edges from $[v]$ to $[w]$, as the following example illustrates.

\begin{example}
Consider automorphic conjugacy class 6.10.
The minimal words $aaaabb$, $aaab\overline{a}b$, and $aab\overline{aa}b$ are vertex representatives for $\Gamma(W)$.
Neither the automorphism $(\{a\}, \overline{b})$ nor its inverse are level on any of these three words.
Let $\phi = (\{b\}, \overline{a})$.
We have $\phi(aaaabb) \equiv aaab\overline{a}b$ and $\phi(aaab\overline{a}b) \equiv aab\overline{aa}b$.
Note that $\phi^{-1}$ is not level on $aaaabb$, so $[aaaabb]$ has outdegree $1$.
On $aab\overline{aa}b$, $\phi$ has the effect of $\phi(aab\overline{aa}b) \equiv ab\overline{aaa}b \equiv \pi(aaab\overline{a}b)$, where $\pi$ is the permutation which maps $a \mapsto \overline{a}$ and $b \mapsto b$, so we have an edge $\pi \phi$ from $aab\overline{aa}b$ to $aaab\overline{a}b$.
Therefore, $\Gamma(W)$ with its vertices labeled is
\[
	\xymatrixcolsep{1cm}
	\xymatrix{
		aaaabb \ar@/^/[r]^\phi & \ar@/^/[l]^{\phi^{-1}} aaab\overline{a}b \ar@/^/[r]^\phi & \ar@/^/[l]^{\phi^{-1}} \ar[l]|{\pi \phi} aab\overline{aa}b
	}.
\]
We suppress brackets here to emphasize that we have fixed a representative of each vertex and that the edge labels are acting on these representatives; in other words, there are no hidden permutations.
As will emerge from the proof of Lemma~\ref{semi-alternating at least 1}, one can think of $\Gamma(W)$ as the path
\[
	\xymatrix{
		aaaabb \ar@/^/[r]^\phi & \ar@/^/[l]^{\phi^{-1}} aaab\overline{a}b \ar@/^/[r]^\phi & \ar@/^/[l]^{\phi^{-1}} aab\overline{aa}b \ar@/^/[r]^\phi & \ar@/^/[l]^{\phi^{-1}} ab\overline{aaa}b \ar@/^/[r]^\phi & \ar@/^/[l]^{\phi^{-1}} b\overline{aaaa}b
	}
\]
folded in half to account for $\pi(aaaabb) \equiv b\overline{aaaa}b$ and $\pi(aaab\overline{a}b) \equiv ab\overline{aaa}b$.
The symmetry in the center word $aab\overline{aa}b$ allows $\pi(aab\overline{aa}b) \equiv aab\overline{aa}b$.
Only three of the four edges between $aab\overline{a}\overline{a}b$ and its neighbors survive the folding, since $\pi$ is applied before $\phi^{-1}$ in $\phi^{-1} \pi(aaab\overline{a}b) \equiv aab\overline{aa}b$, so this automorphism does not contribute an edge to $\Gamma(W)$.
\end{example}

It is also possible for a vertex to have a single loop due to a symmetry in a word.

\begin{example}
If $w = aababaa\overline{b}\overline{b}$ then the automorphism $(\{b\}, \overline{a})$ maps $w$ to the word $(\{b\}, \overline{a})(w) = aabbaa\overline{b}a\overline{b}$.
Let $\pi$ map $a \mapsto a, b \mapsto \overline{b}$; since $\pi (\{b\}, \overline{a})(w) \equiv w$, the vertex $[w]$ has a loop.
However, there is only one loop on $[w]$, since the other three principal one-letter automorphisms are not level on $w$.
This is class~9.43.
\end{example}

The following are our main theorems.
Theorem~\ref{non-root class} is proved in Section~\ref{Non-root classes}, and Section~\ref{Root classes} contains the proofs of Theorems~\ref{non-alternating root class} and \ref{alternating root class}.

\begin{theorem}\label{non-root class}
Let $W$ be a non-root class.
Then $\Gamma(W)$ has one of the following forms.
\begin{itemize}
\item[(P1)]
a simple path
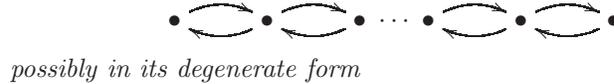
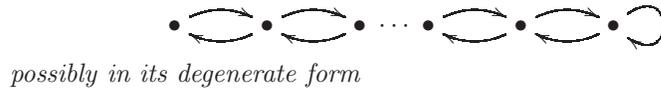
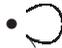
\[
	\xymatrix{
		\bullet \ar@/^/[r] & \ar@/^/[l] \bullet \ar@/^/[r] & \ar@/^/[l] \bullet
	}
	\cdots
	\xymatrix{
		\bullet \ar@/^/[r] & \ar@/^/[l] \bullet \ar@/^/[r] & \ar@/^/[l] \bullet
	}
\]
possibly in its degenerate form
\[
	\xymatrix{
		\bullet
	}
\]
\item[(P2)]
a looped path
\[
	\xymatrix{
		\bullet \ar@/^/[r] & \ar@/^/[l] \bullet \ar@/^/[r] & \ar@/^/[l] \bullet
	}
	\cdots
	\xymatrix{
		\bullet \ar@/^/[r] & \ar@/^/[l] \bullet \ar@/^/[r] & \ar@/^/[l] \bullet \ar@(ur, dr)[] 
	}
\]
possibly in its degenerate form
\[
	\xymatrix{
		\ar@(ur, dr)[] \bullet
	}
\]
\item[(P3)]
a double-edged path
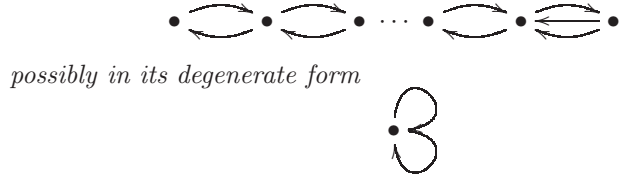
\[
	\xymatrix{
		\bullet \ar@/^/[r] & \ar@/^/[l] \bullet \ar@/^/[r] & \ar@/^/[l] \bullet
	}
	\cdots
	\xymatrix{
		\bullet \ar@/^/[r] & \ar@/^/[l] \bullet \ar@/^/[r] & \ar@/^/[l] \ar[l] \bullet
	}
\]
possibly in its degenerate form
\[
	\xymatrix{
		\bullet \ar@(u, r)[] \ar@(r, d)[]
	}
\bigskip	
\]
\end{itemize}
\end{theorem}

We have referred to the double-looped vertex as a degenerate double-edged path.
This is merely for purposes of convenience; it is not the case that the proof of Theorem~\ref{non-root class} will illustrate a sense in which they are related.
Alternatively, we could have given the double-looped vertex its own label and required that double-edged paths have at least two vertices.
However, then we would also have separated the unlooped vertex and the single-looped vertex from their families, since our proofs in Section~\ref{Non-root classes} treat them separately as well.

\begin{theorem}\label{non-alternating root class}
Let $W$ be a root class with no alternating minimal word.
Then $\Gamma(W)$ is one of the following graphs.
\begin{itemize}
\item[(R1)]
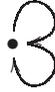
\[
	\xymatrix{
		\bullet \ar@(u, r)[] \ar@(r, d)[]
	}
\]
\item[(R2)]
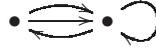
\[
	\xymatrix{
		\bullet \ar[r] \ar@/^/[r] & \ar@/^/[l] \bullet \ar@(ur, dr)[]
	}
\]
\item[(R3)]
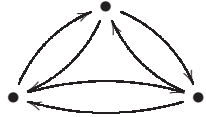
\[
	\xymatrix{
		& {\bullet} \ar@/^/[ld] \ar@/^/[dr] \\
		{\bullet} \ar@/^/[rr] \ar@/^/[ur] & & {\bullet} \ar@/^/[ll] \ar@/^/[ul]
	}
\]
\end{itemize}
\end{theorem}

\begin{theorem}\label{alternating root class}
Let $W$ be a root class containing an alternating minimal word.
Then there is exactly one distinct alternating minimal word modulo $J$ in $W$; denote this word by $w_0$.
Then $\Gamma(W)$ is one of the following graphs.
\begin{itemize}
\item[(R4)]
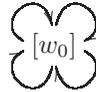
\[
	\xymatrix{
		\ar@(d, l)[] \ar@(l, u)[] \ar@(u, r)[] \ar@(r, d)[] [w_0]
	}
\]
\item[(R5)]
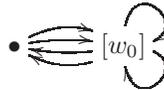
\[
	\xymatrix{
		\bullet \ar@/^.6pc/[r] \ar@/^.2pc/[r] & \ar@/^.6pc/[l] \ar@/^.2pc/[l] [w_0] \ar@(u, r)[] \ar@(r, d)[]
	}
\]
\item[(R6)]
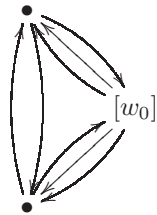
\[
	\xymatrix{
		{\bullet} \ar@/^/[dd] \ar@/^/[dr] & & \\
		& [w_0] \ar@/^/[ul] \ar[ul] \ar@/^/[dl] \ar[dl] \\
		{\bullet} \ar@/^/[uu] \ar@/^/[ur]
	}
\]
\item[(R7)]
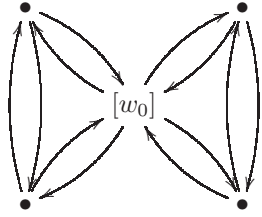
\[
	\xymatrix{
		{\bullet} \ar@/^/[dd] \ar@/^/[dr] & & {\bullet} \ar@/^/[dd] \ar@/^/[dl] \\
		& [w_0] \ar@/^/[ul] \ar@/^/[ur] \ar@/^/[dr] \ar@/^/[dl] \\
		{\bullet} \ar@/^/[uu] \ar@/^/[ur] & & {\bullet} \ar@/^/[uu] \ar@/^/[ul]
	}
\]

\end{itemize}
\end{theorem}

Moreover, each of the ten graph types in Theorems~\ref{non-root class}--\ref{alternating root class} occurs.
See Appendix~\ref{Table of automorphic conjugacy classes} for examples.
Appendix~\ref{Number of automorphic conjugacy classes of each type} lists the number of automorphic conjugacy classes of each graph type for minimal words of length $n \leq 20$.
Since types (P1)--(P3) come in different sizes, Appendix~\ref{Number of paths of each size} lists the number of paths of each size.
Root classes $W$, on the other hand, have bounded size $|V(\Gamma(W))| \in \{1, 2, 3, 5\}$.

From this classification it follows that, with the exception of the double-looped vertex, one can infer from $\Gamma(W)$ whether $W$ is a root class or a non-root class.
Furthermore, if $W$ is a root class then one can infer from $\Gamma(W)$ whether $W$ contains an alternating minimal word or not.

Before embarking on the proofs, we mention a distinguished root word.

\begin{example}
Let $w_0 = (a b \overline{a} \overline{b})^n$.
The image of $w_0$ under $(\{a\}, b)$ is
\[
	(\{a\}, b)(w_0) = ((a b) b (\overline{b} \overline{a}) \overline{b})^n = (a b \overline{a} \overline{b})^n = w_0.
\]
The other three principal automorphisms map $w_0$ either to $(a b \overline{a} \overline{b})^n$ or $(a \overline{b} \overline{a} b)^n$, so $\Gamma(W)$ is (R4).
In fact every class of type (R4) contains $(a b \overline{a} \overline{b})^n$ for some $n \geq 0$, so there is only one such class for each multiple of $4$.
This can be seen as follows.
If $w_0$ is an alternating minimal word of length $4n$ whose class $W$ has size $|V(\Gamma(W))| = 1$, then for each one-letter automorphism $\phi = (\{y\}, x)$ the word $\phi(w_0)$ lies in $[w_0]$ and is therefore alternating.
By Lemma~\ref{cyclic word} we have $0 = (yy)_{\phi(w_0)} = (y\overline{x}y)_{w_0}$, which means that no letter $y$ occurs two letters away from itself.
It follows that $w_0 \equiv \sigma((a b \overline{a} \overline{b})^n)$ for some permutation $\sigma$.
\end{example}

The following lemma is key to the proofs of Theorems~\ref{non-root class}--\ref{alternating root class}.
Under the condition that $w$ is level under a one-letter automorphism, it provides conditions for $w$ to be level under the other principal one-letter automorphisms.

\begin{lemma}\label{level images}
Suppose $w \in C_2$ such that $(\{y\},x)$ is level on $w$.
Then
\begin{enumerate}
\item[(i)] $(\{y\},\overline{x})$ is level on $w$ if and only if $(yy)_w=0$,
\item[(ii)] $(\{x\},y)$ is level on $w$ if and only if $w$ is a root word, and
\item[(iii)] $(\{x\},\overline{y})$ is level on $w$ if and only if $w$ is an alternating root word.
\end{enumerate}
\end{lemma}

\begin{proof}
Since $(\{y\},x)$ is level on $w$, we have
\begin{equation}\label{maltese}
	(y\overline{x})_w = (yx)_w + (yy)_w
\end{equation}
by Lemma~\ref{level automorphism}.
We use this equation frequently in the following.

By Lemma~\ref{level automorphism}, $(\{y\},\overline{x})$ being level on $w$ is equivalent to $(yx)_w = (y\overline{x})_w + (yy)_w$.
Adding this equation to Equation~\eqref{maltese} shows that it is equivalent to $(yy)_w=0$.
This proves~(i).

By Lemma~\ref{level automorphism}, $(\{x\},y)$ being level on $w$ is equivalent to $(x\overline{y})_w=(xy)_w+(xx)_w$, which is equivalent to $(y\overline{x})_w=(yx)_w+(xx)_w$.
Subtracting this from Equation~\eqref{maltese} shows that it is equivalent to $(xx)_w=(yy)_w$, which is equivalent to $w$ being a root word since we also have $(y\overline{x})_w-(yx)_w=(yy)_w$ from Equation~\eqref{maltese}.
This proves~(ii).

Again by Lemma~\ref{level automorphism}, $(\{x\},\overline{y})$ being level on $w$ is equivalent to $(xy)_w=(x\overline{y})_w+(xx)_w$, which is equivalent to $(yx)_w=(y\overline{x})_w+(xx)_w$.
Adding this to Equation~\eqref{maltese} shows that it is equivalent to $0 = (xx)_w + (yy)_w$, which is equivalent to $0 = (xx)_w = (yy)_w = (yx)_w - (y\overline{x})_w$, which is equivalent to $w$ being an alternating root word, giving~(iii).
\end{proof}

Lemma~\ref{level images} already provides enough information to restrict the outdegrees of root word vertices and non-root word vertices.

\begin{corollary}\label{outdegrees}
If $w \in C_2$ is a minimal word that is not a root word, then $\textnormal{outdegree}([w]) \in \{0, 1, 2\}$.
If $w \in C_2$ is a root word, then $\textnormal{outdegree}([w]) \in \{2, 4\}$.
\end{corollary}

\begin{proof}
We have already established that by definition of $\Gamma(W)$ the outdegree of $[w]$ is at most $4$.
Suppose toward a contradiction that the outdegree of $[w]$ is $3$.
Let $(\{y\},x)$ be an automorphism that is level on $w$.
Since the outdegree of each alternating root word is $4$, $w$ is not an alternating root word.
By Lemma~\ref{level images}, the automorphism $(\{x\},\overline{y})$ is therefore not level on $w$, so the other two automorphisms $(\{y\},\overline{x})$ and $(\{x\},y)$ are level on $w$.
By Lemma~\ref{level images}, $(yy)_w = 0$ and $w$ is a root word.
Therefore $(xx)_w = 0$, but this implies that $w$ is alternating and hence an alternating root word, which is a contradiction.
Hence the outdegree of $[w]$ is not $3$.

By Lemma~\ref{level images}, if $w$ is not a root word then additionally the outdegree is not $4$, and if $w$ is a root word then additionally the outdegree is not $1$.

It remains to show that if $w$ is a root word then the outdegree of $[w]$ is at least $1$.
By definition, $w$ is a root word if and only if $|(ab)_w - (a\overline{b})_w| = (aa)_w = (bb)_w$, in which case $(ab)_w - (a\overline{b})_w = (aa)_w$ (and $(\{a\}, \overline{b})$ is level on $w$ by Lemma~\ref{level automorphism}) or $(a\overline{b})_w - (ab)_w = (aa)_w$ (and $(\{a\}, b)$ is level on $w$).
\end{proof}

\section{Non-root classes}\label{Non-root classes}

In this section we prove Theorem~\ref{non-root class} and Theorem~\ref{automorphic conjugacy class size}.
For the duration of this section, fix $x, y \in L_2$ such that $y \notin \{x, \overline{x}\}$.
We say that a word $w$ is \emph{semi-alternating} if $(yy)_w = 0$.
We split the proof of Theorem~\ref{non-root class} into two cases depending on whether the automorphic conjugacy class contains a semi-alternating minimal word.

\begin{lemma}\label{no semi-alternating word}
Let $W$ be a non-root class that contains no semi-alternating minimal word.
Then $\Gamma(W)$ is one of the following graphs.
\begin{itemize}
\item a (P1) path on two vertices
\item a degenerate (P1) path (a single vertex with no edges)
\item a degenerate (P2) path (a single vertex with one loop)
\end{itemize}
\end{lemma}

\begin{proof}
By Lemma~\ref{level images}, every vertex in $\Gamma(W)$ has outdegree at most $1$.
On the other hand, if there is an edge $[w] \to [v]$ then there is an edge $[v] \to [w]$.
Since $\Gamma(W)$ is connected, it follows that $\Gamma(W)$ contains at most $2$ vertices.
If there are $2$ vertices, then $\Gamma(W)$ is a simple path on $2$ vertices.
If there is a single vertex, it can have either one loop or no loops.
\end{proof}

Each of the three possible outcomes in Lemma~\ref{no semi-alternating word} occurs.
One can find examples among words of length $9$.

For $w \in C_2$, define $m_x(w) = \min \{i \geq 0 : (y x^{i} y)_w \geq 1\}$.
Similarly, define $m_{\overline{x}}(w) = \min \{i \geq 0 : (y \overline{x}^{i} y)_w \geq 1\}$.
We adopt the usual convention that $\min \varnothing = \infty$.
Therefore if $(y x^i y)_w = 0$ for all $i \geq 0$ then $m_x(w) = \infty$, for example.
The quantity $m_x(w)$ is a measure of the ``semi-alternatingness'' of $w$.
If $m_x(w) = 0$ then $w$ is not semi-alternating.
If $m_x(w) \geq 1$ then $w$ is semi-alternating and remains so under $m_x(w) - 1$ applications of $(\{y\}, \overline{x})$.

\begin{lemma}\label{m nonzero}
If $w$ is a minimal word, then $1 \leq m_{x}(w) < \infty$ if and only if $1 \leq m_{\overline{x}}(w) < \infty$.
\end{lemma}

\begin{proof}
Consider the one-letter automorphism $\phi = (\{y\}, x)$, which maps $y \mapsto y x$.
This automorphism does not change the distance between $y$ and $\overline{y}$ separated by $x^i$ or $\overline{x}^i$, since for all $i \geq 0$
\begin{align*}
	\phi(y x^i \overline{y}) &= y x^i \overline{y} \\
	\phi(\overline{y} x^i y) &= \overline{x} \overline{y} x^i y x
\end{align*}
and analogously for the inverses of these two words.
On the other hand, $\phi$ does change the distance between a pair of $y$s or a pair of $\overline{y}$s separated by $x^i$ or $\overline{x}^i$, since for all $i \in \mathbb{Z}$
\begin{equation}
	\phi(y x^i y) = y x^{i+1} y x
	\label{minus 1}
\end{equation}
(and analogously for the inverse $\overline{y} x^{-i} \overline{y}$).

Suppose $1 \leq m_{\overline{x}}(w) < \infty$.
Since $w$ is minimal, the image of $w$ under $\phi$ has length at least $|w|$.
Since $\phi$ decreases the distance between the two $y$s in $y \overline{x}^{m_{\overline{x}}(w)} y$ (or the two $\overline{y}$s in $\overline{y} x^{m_{\overline{x}}(w)} \overline{y}$) in $w$, it follows that $\phi$ increases the distance between another pair of $y$s or $\overline{y}$s in $w$.
This can only happen for $y x^j y$ or its inverse for some $j \geq 0$, and since $(yy)_w = 0$ we have $1 \leq m_x(w) < \infty$.

A symmetric argument with the automorphism $(\{y\}, \overline{x})$ shows that if $1 \leq m_x(w) < \infty$ then $1 \leq m_{\overline{x}}(w) < \infty$.
\end{proof}

Since $m_x(w) = 0$ if and only if $m_{\overline{x}}(w) = 0$, it follows from Lemma~\ref{m nonzero} that $m_x(w) = \infty$ if and only if $m_{\overline{x}}(w) = \infty$.

Having proven Lemma~\ref{no semi-alternating word}, it remains to prove Theorem~\ref{non-root class} for classes containing a semi-alternating minimal word.
Lemmas~\ref{semi-alternating infinity} and \ref{semi-alternating at least 1} address the cases $m_x(w) = \infty$ and $1 \leq m_x(w) < \infty$ for the semi-alternating word $w$.
The following lemma shows that a vertex containing a semi-alternating word has outdegree at least $2$.

\begin{lemma}\label{semi-alternating outdegree}
Let $w$ be a semi-alternating minimal word of length $|w| \geq 2$.
Then $(\{y\}, x)$ and $(\{y\}, \overline{x})$ are level on $w$.
\end{lemma}

\begin{proof}
Toward a contradiction, assume that neither $(\{y\}, x)$ nor $(\{y\}, \overline{x})$ is level on $w$.
If $\phi = (\{y\}, x)$ increases the length of $w$, then $\phi$ causes more additions than cancellations in $w$; as in Lemma~\ref{level automorphism}, this implies $(y\overline{x})_w < (yx)_w + (yy)_w$.
Symmetrically, $|(\{y\}, \overline{x})(w)| > |w|$ implies $(yx)_w < (y\overline{x})_w + (yy)_w$.
It follows that $-(yy)_w < (yx)_w - (y\overline{x})_w < (yy)_w$, so $(yy)_w \neq 0$, contradicting the assumption that $w$ is semi-alternating.
Therefore $(\{y\}, x)$ or $(\{y\}, \overline{x})$ is level on $w$.
By Lemma~\ref{level images}, both are.
\end{proof}

\begin{lemma}\label{semi-alternating infinity}
Let $W$ be a non-root class containing a minimal word $w$ such that $m_x(w) = \infty$.
Then $\Gamma(W)$ is a degenerate (P3) path (a single vertex with two loops).
\end{lemma}

\begin{proof}
By Lemma~\ref{semi-alternating outdegree}, $\phi = (\{y\}, x)$ and $\phi^{-1} = (\{y\}, \overline{x})$ are level on $w$.
By Lemma~\ref{level images}, $\phi$ and $\phi^{-1}$ are the only one-letter automorphisms that are level on $w$.
Since $m_{x}(w) = \infty$ and $m_{\overline{x}}(w) = \infty$, $w$ consists of overlapping subwords of the form $y^e x^i y^{-e}$ for $e \in \{1, -1\}$ and $i \in \mathbb{Z} \setminus \{0\}$.
Since the distance between $y^e$ and $y^{-e}$ is fixed by $\phi$ and by $\phi^{-1}$, $w$ is fixed by $\phi$ and by $\phi^{-1}$, so $[w]$ has two loops.

Suppose that $(\{x\}, y)$ is level on $w$.
By Lemma~\ref{level automorphism}, $(xx)_w = (x \overline{y})_w - (xy)_w = (x \overline{y})_w - (yx)_w$.
This difference is equal to $0$ since $m_x(w) = \infty$ implies that the subwords $x\overline{y}$ and $yx$ occur in pairs in $w$ and similarly the subwords $y\overline{x}$ and $\overline{x}\overline{y}$ occur in pairs.
But $(xx)_w = 0$ implies that $w$ is an alternating minimal word and hence a root word by Theorem~\ref{alternating minimal word}, contradicting one of our assumptions.
Therefore $(\{x\}, y)$ is not level on $w$.
Similarly, $(\{x\}, \overline{y}) = (\{x\}, y)^{-1}$ is not level on $w$.
\end{proof}

We use the following result in the proof of Lemma~\ref{semi-alternating at least 1}.

\begin{lemma}\label{semi-alternating at least 1 assist}
Let $w$ be a minimal word such that $1 \leq m_x(w) < \infty$ and $\phi = (\{y\}, x)$ is level on $w$.
Then $\phi$ is level on $\phi(w)$ if and only if $2 \leq m_{\overline{x}}(w) < \infty$.
\end{lemma}

\begin{proof}
Since $\phi^{-1}$ is level on $\phi(w)$, we see by Lemma~\ref{level images} that $\phi$ is level on $\phi(w)$ if and only if $(yy)_{\phi(w)} = 0$.
By Lemma~\ref{cyclic word}, $(yy)_{\phi(w)} = (y \overline{x} y)_w$.
Since $1 \leq m_{\overline{x}}(w) < \infty$ by assumption, $(y \overline{x} y)_w = 0$ if and only if $2 \leq m_{\overline{x}}(w) < \infty$.
\end{proof}

\begin{lemma}\label{semi-alternating at least 1}
Let $W$ be a non-root class containing a minimal word $w$ such that $1 \leq m_{x}(w) < \infty$.
Then $\Gamma(W)$ is a (P1), (P2), or (P3) path with at least $2$ vertices.
\end{lemma}

\begin{proof}
Lemma~\ref{semi-alternating outdegree} and Lemma~\ref{level images} imply that $\phi = (\{y\}, x)$ and its inverse are the only one-letter automorphisms that are level on $w$.
Recall that $J$ is the subgroup of $\Aut F_2$ generated by inner automorphisms and permutations.
Let
\[
	W' = \{\phi^j(w) : -m_{x}(w) \leq j \leq m_{\overline{x}}(w) \}.
\]
Claim:
$W' \subset W$, and for each minimal $v \in W$ the set $W'$ contains a minimal word equivalent to $v$ modulo $J$.
Note that in $W'$ we may have pairs of words that are equivalent modulo $J$.

Toward this claim, we first show that for $-m_{x}(w) \leq j \leq m_{\overline{x}}(w)$ the word $\phi^j(w)$ is minimal, and for $-m_{x}(w) < j < m_{\overline{x}}(w)$ we also show that $\phi^j(w)$ is semi-alternating.
We work by induction on $j$.
For $j = 0$, we have by hypothesis that $w$ is minimal and semi-alternating.
Now, suppose that $\phi^j (w)$ is minimal and semi-alternating for some $0 \leq j < m_{\overline{x}}(w)$.
Then $\phi^{-1}$ is level on $\phi^j (w)$, so since $\phi^j(w)$ is semi-alternating we have that $\phi$ is level on $\phi^j(w)$ by Lemma~\ref{level images}.
Thus, $\phi^{j+1}(w)$ is minimal.
It remains to show that if $j + 1 < m_{\overline{x}}(w)$ then $\phi^{j+1}(w)$ is semi-alternating.
In this case, by Equation~\eqref{minus 1} we have $m_{\overline{x}}(\phi^j(w)) = m_{\overline{x}}(w) - j \geq 2$, so Lemma~\ref{semi-alternating at least 1 assist} yields that $\phi^{j+1}(w)$ is semi-alternating.
A symmetric argument with $\phi^{-1}$ establishes the cases $-m_{x}(w) \leq j \leq 0$.

In fact $\phi^{-m_x(w)} (w)$ and $\phi^{m_{\overline{x}}(w)} (w)$ are not semi-alternating, since by Equation~\eqref{minus 1} $m_{\overline{x}}(\phi^{m_{\overline{x}}(w)} (w)) = m_{\overline{x}}(w) - m_{\overline{x}}(w) = 0$.
Similarly, $m_x((\phi^{-1})^{m_x(w)}(w)) = 0$.
This means that $\phi^{-m_x(w)}(w)$ and $\phi^{m_{\overline{x}}(w)}(w)$ each have at most one level one-letter automorphism (again by Lemma~\ref{level images}), and in fact $\phi$ and $\phi^{-1}$ respectively are level on these words.

For each $-m_x(w) \leq j \leq m_{\overline{x}}(w)$ we have determined the images of $\phi^j(w)$ under all level automorphisms.
Since $V(\Gamma(W))$ is connected by level one-letter automorphisms, $W'$ projects onto $V(\Gamma(W))$ and the claim follows.

In order to determine $\Gamma(W)$ from $W'$, we need to consider the possibility that some words have been listed in $W'$ more than once up to equivalence under $\sim_J$.
For the two endpoint words $\phi^{-m_x(w)}(w)$ and $\phi^{m_{\overline{x}}(w)}(w)$ we have
\begin{align*}
	m_{\overline{x}}(\phi^{-m_x(w)}(w)) &= m_{\overline{x}}(w) + m_x(w) \geq 2, \\
	m_x(\phi^{m_{\overline{x}}(w)}(w)) &= m_x(w) + m_{\overline{x}}(w) \geq 2.
\end{align*}
It follows that for $u \in \{\phi^{-m_x(w)}(w), \phi^{m_{\overline{x}}(w)}(w)\}$ we have $(xx)_u \geq 1$.
Since $u$ is also not semi-alternating, $u$ is not the image of $\phi^j(w)$ under an element of $J$.
Therefore, at least one minimal word in $W$ is semi-alternating, and at least one but at most two distinct minimal words modulo $J$ in $W$ are not semi-alternating.
So $\Gamma(W)$ is a connected directed graph with either one or two vertices having outdegree $1$ and all other vertices having outdegree $2$.
Since an edge from $[v_i]$ to $[v_j]$ in $\Gamma(W)$ implies an edge from $[v_j]$ to $[v_i]$, $\Gamma(W)$ is one of the paths claimed.
\end{proof}

We have completed the proof of Theorem~\ref{non-root class}.
The following examples illustrate the path (P1) of Lemma~\ref{semi-alternating at least 1}.

\begin{example}
Class~9.81 contains the word $w = aabab\overline{a}b\overline{a}b$, which for $y = b$ is semi-alternating.
We have $m_a(w) = 1$ and $m_{\overline{a}}(w) = 1$, so $\Gamma(W)$ for this class is
\[
	\xymatrix{
		aaabaabbb \ar@/^/[r]^\phi & \ar@/^/[l]^{\phi^{-1}} aabab\overline{a}b\overline{a}b \ar@/^/[r]^\phi & \ar@/^/[l]^{\phi^{-1}} abb\overline{a}\overline{a}b\overline{a}\overline{a}b
	}
\]
where $\phi = (\{b\}, \overline{a})$.
Observe that $\phi$ shrinks subwords $b a^i b$ (and their inverses), extends subwords $b \overline{a}^i b$ (and their inverses), and leaves subwords $b a^{\pm i} \overline{b}$ (and their inverses) fixed.
Vertices with outdegree $1$ have $(bb)_w \geq 1$.
In each subword $bb$ of $aaabaabbb$ the automorphism $\phi$ introduces $\overline{a}$.
After applying $\phi$ twice, the subword $baab$ becomes $bb$, so further applications of $\phi$ produce words that are not minimal.
\end{example}

\begin{example}
If we begin with a minimal word with $(bb)_w = 1$ rather than $(bb)_w = 2$, then the automorphic conjugacy class can be larger since the word grows at only one position rather than two.
For example, consider the word $aaa{ba\overline{b}ab}b$ belonging to class~9.97.
Its graph $\Gamma(W)$ is
\[
	\xymatrix{
		aaa{ba\overline{b}ab}b \ar@/^/[r]^\phi & \ar@/^/[l]^{\phi^{-1}} aa{ba\overline{b}ab}\overline{a}b \ar@/^/[r]^\phi & \ar@/^/[l]^{\phi^{-1}} a{ba\overline{b}ab}\overline{a}\overline{a}b \ar@/^/[r]^\phi & \ar@/^/[l]^{\phi^{-1}} {ba\overline{b}ab}\overline{a}\overline{a}\overline{a}b
	}
\]
where again $\phi = (\{b\}, \overline{a})$.
\end{example}

The automorphic conjugacy classes that are most relevant for Theorem~\ref{automorphic conjugacy class size} are those addressed by Theorem~\ref{non-root class}.
Therefore we now give a proof of Theorem~\ref{automorphic conjugacy class size}, even though Theorems~\ref{non-alternating root class} and \ref{alternating root class} on which it depends will be proved in Section~\ref{Root classes}.

\begin{proof}[Proof of Theorem~\ref{automorphic conjugacy class size}]
If $W$ is a root class whose minimal words have length $n \geq 9$, then in fact $n \geq 12$ by Theorem~\ref{divisibility by four}; by Theorems~\ref{non-alternating root class} and \ref{alternating root class}, $|V(\Gamma(W))| \leq 5 \leq n - 5$.

Therefore let $W$ be a non-root class whose minimal words have length $n \geq 9$.
We may assume that $W$ contains a minimal word $v$ with $1 \leq m_{x}(v) < \infty$, since otherwise $|V(\Gamma(W))| \leq 2$ by Lemmas~\ref{no semi-alternating word} and \ref{semi-alternating infinity}.
By the proof of Lemma~\ref{semi-alternating at least 1},
\begin{align*}
	|V(\Gamma(W))|
	&\leq m_x(v) + 1 + m_{\overline{x}}(v) \\
	&\leq 1 + \max \{i : \text{$x^i$ appears in a minimal word in $W$}\}.
\end{align*}
Therefore it suffices to show that if $x^i$ appears in a minimal word $w$ of length $n$ and $ n-6 < i \leq n - 1$ then $|V(\Gamma(W))| \leq n - 5$.

By applying a permutation, we may assume $x = a$ and $y = b$, so $w \equiv a^{n-1} b$ or $w \equiv a^{n-k} b u b$ for some subword $u$ of length $k - 2 \leq 3$.
The word $a^{n-1} b$ is not minimal, so it suffices to consider $a^{n-k} b u b$.
There are sufficiently few possibilities for $u$ that we simply check them all.

If $u$ is the empty word, then $w \equiv a^{n-2} bb$.
This word is minimal, and its graph $\Gamma(W)$ is of type (P2) for odd $n$ and of type (P3) for even $n$.
The number of vertices in $\Gamma(W)$ is $\lfloor n/2 \rfloor$, which satisfies $\lfloor n/2 \rfloor \leq n - 5$ for $n \geq 9$, 

There are $3$ words of length $1$ to check.
If $u = a$ then $w$ is not minimal.
If $u = \overline{a}$ then $w \sim a^{n-2} bb$ so we have already shown that the graph has at most $n - 5$ vertices.
If $u = b$ then $w$ is minimal, and $\Gamma(W)$ is (P1) of size $1$.

There are $7$ words of length $2$ to check:
\[
\begin{array}{cc}
u & \Gamma(W) \\ \hline
aa				& \text{\gray{not minimal}} \\
ab				& \text{\gray{not minimal}} \\
ba				& \text{\gray{not minimal}} \\
bb				& \text{(P1) of size $1$} \\
b\overline{a}			& \text{(P1) of size $1$} \\
\overline{a}b			& \text{(P1) of size $1$} \\
\overline{a}\overline{a}	& \text{(P2) or (P3) of size $\lfloor n/2 \rfloor$}
\end{array}
\]

Finally, there are $21$ words of length $3$ to check:
\[
\begin{array}{cc}
u & \Gamma(W) \\ \hline
aaa					& \text{\gray{not minimal}} \\
aab					& \text{\gray{not minimal}} \\
aba					& \text{\gray{not minimal}} \\
abb					& \text{(P1) of size $2$} \\
ab\overline{a}				& \text{\gray{not minimal}} \\
a\overline{b}a				& \text{\gray{not minimal}} \\
a\overline{b}\overline{a}		& \text{\gray{not minimal}} \\
baa					& \text{\gray{not minimal}} \\
bab					& \text{(P1) of size $2$} \\
bba					& \text{(P1) of size $2$} \\
bbb					& \text{(P1) of size $1$}
\end{array}
\qquad \qquad
\begin{array}{cc}
u & \Gamma(W) \\ \hline
bb\overline{a}				& \text{(P1) of size $1$} \\
b\overline{a}b				& \text{(P1) of size $1$} \\
b\overline{a}\overline{a}		& \text{(P1) of size $1$} \\
\overline{a}ba				& \text{\gray{not minimal}} \\
\overline{a}bb				& \text{(P1) of size $1$} \\
\overline{a}b\overline{a}		& \text{\gray{not minimal}} \\
\overline{a}\overline{a}b		& \text{(P1) of size $1$} \\
\overline{a}\overline{a}\overline{a}	& \text{(P2) or (P3) of size $\lfloor n/2 \rfloor$} \\
\overline{a}\overline{b}a		& \text{\gray{not minimal}} \\
\overline{a}\overline{b}\overline{a}	& \text{\gray{not minimal}} \\
{}					& {}
\end{array}
\]

Hence $|V(\Gamma(W))| \leq n - 5$ for all minimal words $a^{n-k} b u b$ of length $n$ where $2 \leq k \leq 5$, and the statement follows.
\end{proof}

Theorem~\ref{automorphic conjugacy class size} is sharp in the sense that for every $n \geq 9$ there exists an automorphic conjugacy class $W$ with minimal words of length $n$ such that $|V(\Gamma(W))| = n - 5$.
For example, the class containing $a^{n-6}ba\overline{b}abb$ is a (P1) class with distinct vertex representatives $a^{n-6-j}ba\overline{b}ab\overline{a}^{j}b$ for $0 \leq j \leq n-6$.
There appear to be $5$ such (P1) classes for each $n \geq 9$; see Section~\ref{Enumeration} and Appendix~\ref{Number of paths of each size}.

As can be observed from the data in Appendix~\ref{Number of paths of each size}, when $n$ is odd the double-edged path occurs only in its degenerate form.

\begin{corollary}
Let $W$ be a non-root class containing a minimal word of odd length such that $\Gamma(W)$ is of type (P3).
Then $|V(\Gamma(W))| = 1$.
\end{corollary}

\begin{proof}
Let $[w]$ be the endpoint with outdegree $2$ of a nondegenerate (P3) graph.
By the proof of Lemma~\ref{semi-alternating at least 1}, $\phi(w) \sim_J \phi^{-1}(w)$ for some one-letter automorphism $\phi = (\{y\}, x)$ that is level on $w$.
Write $\pi \phi(w) \equiv \phi^{-1}(w)$ for some permutation $\pi$.
Since $w$ is semi-alternating and $(xx)_w \neq 0$, $\pi(x) \in \{x, \overline{x}\}$ and $\pi(y) \in \{y, \overline{y}\}$.
If $\pi$ maps $x \mapsto x, y \mapsto y$ or $x \mapsto \overline{x}, y \mapsto \overline{y}$, then by Lemma~\ref{permutation} $\phi \pi(w) \equiv \phi^{-1}(w)$, so $\pi(w) \equiv \phi^{-2}(w)$, which contradicts $[w]$ being an endpoint.
Therefore $\pi$ maps $x \mapsto \overline{x}, y \mapsto y$ or $x \mapsto x, y \mapsto \overline{y}$.
By Lemma~\ref{permutation}, $\phi^{-1} \pi(w) \equiv \phi^{-1}(w)$, so $w$ has a symmetry $\pi(w) \equiv w$.
Let $k \geq 1$ be minimal such that $w = \rho^k \pi(w)$, where $\rho$ is rotation to the right by one character.
Let $u$ be the prefix of $w$ of length $k$.
Then $w = u \cdot \pi(u) \cdot \pi^2(u) \cdot \pi^3(u) \cdots \pi^{-1}(u)$.
Since $\pi$ has order $2$, we have $w = (u \cdot \pi(u))^{|w|/(2k)}$ and $|w|$ is even.
\end{proof}

\section{Root classes}\label{Root classes}

In this section we prove Theorems~\ref{non-alternating root class} and \ref{alternating root class}, establishing the structure of root classes.
For this, we need a lemma concerning the composition of two one-letter automorphisms.
Note that we compose functions from right to left, as in Section~\ref{Introduction}.

\begin{lemma}\label{triangle}
Let $x, y \in L_2$ with $y \notin \{x, \overline{x}\}$.
Let $\pi$ be the permutation which maps $x \mapsto \overline{y}$ and $y \mapsto x$.
Then
\[
	(\{\overline{x}\}, y) \cdot (\{y\}, x) = \pi \cdot (\{x, \overline{x}\}, y) \cdot (\{\overline{x}\}, \overline{y}).
\]
\end{lemma}

\begin{proof}
One checks that both sides map $x \mapsto \overline{y} x$ and $y \mapsto x$.
\end{proof}

A consequence of Lemma~\ref{triangle} is that $[(\{\overline{x}\}, y) (\{y\}, x) (w)] = [(\{\overline{x}\}, \overline{y}) (w)]$ for all $w \in C_2$.
That is, the vertex $[(\{\overline{x}\}, y) (\{y\}, x) (w)]$ is a neighbor of $[w]$ in $\Gamma(W)$.

Now we determine the structure of root classes with no alternating word.

\begin{proof}[Proof of Theorem~\ref{non-alternating root class}]
Let $W$ be a root class with no alternating minimal word.

By Corollary~\ref{outdegrees}, the outdegree of a root word vertex $[w]$ is either $2$ or $4$.
If $w$ is not alternating, then by Lemma~\ref{level images} there are only two level one-letter automorphisms on $w$.
Therefore every vertex in $\Gamma(W)$ has outdegree $2$.

We show that any two distinct vertices in $\Gamma(W)$ are neighbors.
Suppose that $u, v, w \in W$ are minimal words such that $v \equiv \phi(w)$ and $u \equiv \psi(v) \equiv \psi \phi(w)$ for some one-letter automorphisms $\phi = (\{y\}, x)$ and $\psi$.
We want to show that either $[w] = [u]$ or $[w]$ is connected to $[u]$ by a one-letter automorphism.
This will then imply that any two vertices that are connected by a sequence of one-letter automorphisms are either the same vertex or are in fact connected by a single one-letter automorphism.

We know that $\phi^{-1} = (\{y\}, \overline{x})$ is level on $v$.
Since $v$ is a root word which is not alternating, we have $(xx)_v = (yy)_v \neq 0$ and therefore Lemma~\ref{level images} implies that $\phi^{-1}$ and $(\{\overline{x}\}, y)$ are the only (distinct modulo $\Inn F_2$) one-letter automorphisms that are level on $v$.
Since $\psi$ is level on $v$, $\psi$ is equivalent modulo $\Inn F_2$ to either $\phi^{-1}$ or $(\{\overline{x}\}, y)$.
There are therefore two cases.
If $\psi$ is equivalent to $\phi^{-1}$, then we have $w \equiv u$.
If instead $\psi$ is equivalent to $\psi' = (\{\overline{x}\}, y)$, then by Lemma~\ref{triangle} we have $\psi \phi(w) \equiv \psi' \phi(w) = \pi (\{x, \overline{x}\}, y) (\{\overline{x}\}, \overline{y})(w)$, where $\pi$ is the permutation which maps $x \mapsto \overline{y}$ and $y \mapsto x$; this implies that $[w]$ is connected to $[\psi \phi(w)] = [u]$ by a one-letter automorphism.

We have shown that if $w$ and $u$ are minimal words in $W$, then $[w] = [u]$ or $[w]$ and $[u]$ are neighbors.
Since the outdegree of each vertex in $\Gamma(W)$ is $2$, this implies that there are at most three vertices in $\Gamma(W)$.
If $|V(\Gamma(W))| = 1$, then $\Gamma(W)$ is (R1), a single vertex with two loops.
If $|V(\Gamma(W))| = 3$, then $\Gamma(W)$ is (R3), a bi-directed $3$-cycle.
Otherwise, $|V(\Gamma(W))| = 2$.
Let $[w]$ and $[\phi(w)]$ be the two vertices of $\Gamma(W)$.
There is a directed edge from $[w]$ to $[\phi(w)]$ and another from $[\phi(w)]$ to $[w]$, so it suffices to determine the other two edges.
As above, $(\{\overline{x}\}, y)$ is level on $\phi(w)$ and not equivalent modulo $\Inn F_2$ to $\phi^{-1}$, so this automorphism contributes an edge from $[\phi(w)]$ to $[(\{\overline{x}\}, y)(w)]$, which is one of the two vertices.
By Lemma~\ref{triangle}, there is a directed edge from $[w]$ to $[(\{\overline{x}\}, y)(w)]$.
Therefore the other two edges point to the same vertex, and $\Gamma(W)$ is (R2).
\end{proof}

\begin{example}
Let $W$ be class~8.37, whose graph is (R3).
Let $\pi$ be the permutation mapping $a \mapsto b, b \mapsto \overline{a}$.
Write $\phi_{yx} = (\{y\}, x)$.
Then $\Gamma(W)$ is the following graph, where an edge $w \to v$ labeled $\phi$ satisfies $\phi(w) \equiv v$.
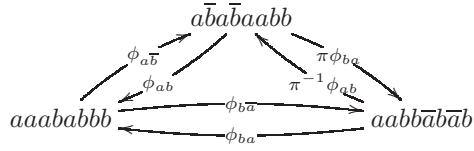
\[
	\xymatrix{
		& a\overline{b}a\overline{b}aabb \ar@/^/[dl]|{\phi_{ab}} \ar@/^/[dr]|{\pi \phi_{ba}} \\
		aaababbb \ar@/^/[rr]|{\phi_{b\overline{a}}} \ar@/^/[ur]|{\phi_{a\overline{b}}} & & aabb\overline{a}b\overline{a}b \ar@/^/[ll]|{\phi_{ba}} \ar@/^/[ul]|{\pi^{-1} \phi_{ab}}
	}
\]
\end{example}

Now we start with alternating words.
We need several lemmas.

\begin{lemma}\label{alternating fixed}
Suppose $w_0$ is an alternating minimal word and $\phi$ is a one-letter automorphism such that $\phi(w_0)$ is an alternating minimal word.
Then $\phi(w_0) = w_0$.
\end{lemma}

\begin{proof}
Write $\phi = (\{y\}, x)$.
Since $\phi(w_0)$ is alternating, we have $(yy)_{\phi(w_0)} = 0$, so $(y\overline{x}y)_{w_0} = 0$ by Lemma~\ref{cyclic word}.
The only length-$2$ subwords that cause cancellations under $\phi$ are $y\overline{x}$ and $x\overline{y}$.
Since $(y\overline{x}y)_{w_0} = 0$ and $w_0$ is alternating, every $y\overline{x}$ in $w_0$ appears in $y\overline{x}\overline{y}$ and every $x\overline{y}$ appears in $yx\overline{y}$.
But $\phi(y\overline{x}\overline{y}) = y\overline{x}\overline{y}$ and $\phi(yx\overline{y}) = yx\overline{y}$, so $\phi$ causes no cancellations in $w_0$.
Since all one-letter automorphisms are level on $w_0$ by Theorem~\ref{alternating minimal word}, $\phi$ also causes no additions in $w_0$.
Therefore $\phi(w_0) = w_0$.
\end{proof}

For the rest of this section, denote $\phi_1 = (\{y\}, x)$, $\phi_2 = \phi_1^{-1} = (\{y\}, \overline{x})$, $\phi_3 = (\{x\}, y)$, and $\phi_4 = \phi_3^{-1} = (\{x\}, \overline{y})$.
These are four principal one-letter automorphisms, and they are distinct modulo $\Inn F_2$.
In this notation, Lemma~\ref{triangle} implies that $[\phi_4 \phi_1(w)] = [\phi_3(w)]$.
We record this in the following corollary, along with analogous statements obtained by applying permutations to $L_2$.

\begin{corollary}\label{triangle corollary}
For $w \in C_2$,
\begin{align*}
	[\phi_2 \phi_3(w)] &= [\phi_1(w)] \\
	[\phi_1 \phi_4(w)] &= [\phi_2(w)] \\
	[\phi_4 \phi_1(w)] &= [\phi_3(w)] \\
	[\phi_3 \phi_2(w)] &= [\phi_4(w)].
\end{align*}
\end{corollary}

The statements of the next three lemmas are all of the same form.
They determine the neighborhood of a vertex containing an alternating minimal word.
They form the bulk of the proof of Theorem~\ref{alternating root class}.
Recall from Theorem~\ref{alternating minimal word} that all one-letter automorphisms are level on alternating minimal words.

\begin{lemma}\label{s1 = s2}
Let $w_0$ be an alternating minimal word such that $[\phi_1(w_0)] = [\phi_2(w_0)]$ for some $x, y \in L_2$ with $y \notin \{x, \overline{x}\}$.
Then $[\phi_3(w_0)] = [\phi_4(w_0)]$.
\end{lemma}

\begin{proof}
By Theorem~\ref{alternating minimal word}, $\phi_1$ is level on $w_0$.
Since $\phi_2 = \phi_1^{-1}$, $\phi_2$ is level on $\phi_1(w_0)$.
By Lemma~\ref{level images}, $\phi_4$ is also level on $\phi_1(w_0)$.
Let us compute the neighbors of $\phi_1(w_0)$ under $\phi_2$ and $\phi_4$.
We have $[\phi_2 \phi_1(w_0)] = [w_0]$, and Corollary~\ref{triangle corollary} implies that $[\phi_4 \phi_1(w_0)] = [\phi_3(w_0)]$.
Similarly, the images of $\phi_2(w_0)$ under $\phi_1$ and $\phi_3$ are $[\phi_1 \phi_2(w_0)] = [w_0]$ and $[\phi_3 \phi_2(w_0)] = [\phi_4(w_0)]$.

If $\phi_1(w_0)$ is alternating, then $\phi_1(w_0) = w_0$ by Lemma~\ref{alternating fixed}.
Then $[\phi_4 \phi_1(w_0)] = [\phi_3(w_0)]$ implies $[\phi_4(w_0)] = [\phi_3(w_0)]$ as desired.

If $\phi_1(w_0)$ is not alternating, then by Lemma~\ref{level images} the outdegree of $[\phi_1(w_0)]$ is $2$.
Since we have shown that $[w_0]$, $[\phi_3(w_0)]$, and $[\phi_4(w_0)]$ are all neighbors of $[\phi_1(w_0)]$, it follows that two of these three vertices are equal.
If $[\phi_3(w_0)] = [\phi_4(w_0)]$, we are finished.
If $[w_0] = [\phi_3(w_0)]$ or $[w_0] = [\phi_4(w_0)]$, then we see that $\phi_3(w_0)$ or $\phi_4(w_0)$ is alternating; in either case Lemma~\ref{alternating fixed} gives $\phi_3(w_0) = w_0 = \phi_4(w_0)$.
\end{proof}

\begin{lemma}\label{s1 = s3}
Let $w_0$ be an alternating minimal word such that $[\phi_1(w_0)] = [\phi_3(w_0)]$ for some $x, y \in L_2$ with $y \notin \{x, \overline{x}\}$.
Then $[\phi_2(w_0)] = [\phi_4(w_0)]$.
\end{lemma}

\begin{proof}
By the definition of a root word, $(yy)_{\phi_1(w_0)} = (xx)_{\phi_1(w_0)}$; rewriting each side using Lemma~\ref{cyclic word} gives
\[
	(y\overline{x}y)_{w_0} = (yxy)_{w_0} + (yxx)_{w_0} + (xxy)_{w_0} + (xxx)_{w_0}.
\]
Since $w_0$ is alternating, this equation becomes $(y\overline{x}y)_{w_0} = (yxy)_{w_0}$.
Symmetrically, since $\phi_3(w_0)$ is a root word, we have $(x\overline{y}x)_{w_0} = (xyx)_{w_0}$.

Let $\pi$ be a permutation such that $\phi_1(w_0) \equiv \pi \phi_3 (w_0)$.
Then $(xx)_{\phi_1(w_0)} = (xx)_{\pi \phi_3 (w_0)} = (xx)_{\phi_3 (w_0)}$, so
\[
	(y\overline{x}y)_{w_0} = (yxy)_{w_0} = (x\overline{y}x)_{w_0} = (xyx)_{w_0}.
\]
For six of the eight possible permutations $\pi$, we show that these four expressions are equal to $0$.
For these $\pi$, this will imply that no letter occurs two letters away from itself in $w_0$, so $w_0 \equiv \sigma((a b \overline{a} \overline{b})^n)$ for some permutation $\sigma$.
As already stated in Section~\ref{The graph}, for this word we have $[\phi(w_0)] = [w_0]$ for each one-letter automorphism $\phi$.

If $\pi$ maps $x \mapsto x, y \mapsto y$ or $x \mapsto \overline{x}, y \mapsto \overline{y}$, consider $(yxy)_{\phi_1(w_0)} = (yxy)_{\pi \phi_3(w_0)}$.
Then $(yxy)_{\phi_1(w_0)} = (yxy)_{\phi_3(w_0)}$, and rewriting each side gives
\[
	(yy)_{w_0} = (yxy)_{w_0} + (yxx)_{w_0} + (xxy)_{w_0} + (xxx)_{w_0},
\]
which simplifies to $0 = (yxy)_{w_0}$ because $w_0$ is alternating.

If $\pi$ maps $x \mapsto \overline{x}, y \mapsto y$ or $x \mapsto x, y \mapsto \overline{y}$, then consider $(y\overline{x}y)_{\phi_1(w_0)} = (y\overline{x}y)_{\pi \phi_3(w_0)}$.
Since $(y\overline{x}y)_{\pi \phi_3(w_0)} = (yxy)_{\phi_3(w_0)}$, the right side is the same as before, and we obtain
\[
	(y\overline{x}\overline{x}y)_{w_0} = (yxy)_{w_0} + (yxx)_{w_0} + (xxy)_{w_0} + (xxx)_{w_0},
\]
which simplifies to $0 = (yxy)_{w_0}$.

If $\pi$ maps $x \mapsto \overline{y}, y \mapsto x$ or $x \mapsto y, y \mapsto \overline{x}$, use Lemma~\ref{permutation} to write $\phi_1(w_0) \equiv \pi \phi_3(w_0) = (\{\pi(x)\}, \pi(y)) \pi(w_0)$.
In either case, we obtain $\phi_1(w_0) \equiv \phi_2 \pi(w_0)$ (where for the permutation $x \mapsto \overline{y}, y \mapsto x$ we have used Equation~\eqref{inner automorphism}).
Hence $\phi_1^2(w_0) \equiv \pi(w_0)$, and $\phi_1^2(w_0)$ is alternating.
In particular, $(yxxxy)_{\phi_1^2(w_0)} = 0$, and this implies $(yxy)_{w_0} = 0$.

Two permutations remain to be considered.
Let $\pi$ map $x \mapsto y, y \mapsto x$ or $x \mapsto \overline{y}, y \mapsto \overline{x}$.
Lemma~\ref{permutation} gives $\phi_1(w_0) \equiv \pi \phi_3(w_0) = (\{\pi(x)\}, \pi(y)) \pi(w_0) \equiv \phi_1 \pi(w_0)$.
Hence $w_0 \equiv \pi(w_0)$.
We show that the only alternating minimal word satisfying this equation is the empty word.
Assume toward a contradiction that $w_0$ is nonempty.
Let $k \geq 1$ be minimal such that $w_0 = \rho^k \pi(w_0)$, where $\rho$ is rotation to the right by one character.
Let $u$ be the prefix of $w_0$ of length $k$.
Then $w_0 = u \cdot \pi(u) \cdot \pi^2(u) \cdot \pi^3(u) \cdots \pi^{-1}(u)$.
Since $\pi$ has order $2$, we have $w_0 = (u \cdot \pi(u))^{|w_0|/(2k)}$ and $|w_0|$ is divisible by $2 k$.
Since $w_0$ is alternating and $\pi(x) \in \{y, \overline{y}\}$, $k$ is odd.
Since $w_0$ is a root word, it follows that $u \cdot \pi(u)$ is a root word.
By Theorem~\ref{divisibility by four}, $|u \cdot \pi(u)| = 2 k$ is divisible by $4$, which contradicts $k$ being odd.
\end{proof}

As we have just seen, $(a b \overline{a} \overline{b})^n$ is essentially the only alternating minimal word satisfying $[\phi_1(w_0)] = [\phi_3(w_0)]$.
However, the equation $[\phi_1(w_0)] = [\phi_4(w_0)]$, which is the subject of the following lemma, has additional solutions.
For example, $abab\overline{a}b\overline{a}\overline{b}\overline{a}\overline{b}a\overline{b}$ is a solution.

\begin{lemma}\label{s1 = s4}
Let $w_0$ be an alternating minimal word such that $[\phi_1(w_0)] = [\phi_4(w_0)]$ for some $x, y \in L_2$ with $y \notin \{x, \overline{x}\}$.
Then $[\phi_2(w_0)] = [\phi_3(w_0)]$.
\end{lemma}

\begin{proof}
As in the proof of Lemma~\ref{s1 = s3}, one can show that
\[
	(y\overline{x}y)_{w_0} = (yxy)_{w_0} = (x\overline{y}x)_{w_0} = (xyx)_{w_0}.
\]
Write $\phi_1(w_0) \equiv \pi \phi_4(w_0)$.
For six of the eight possible permutations $\pi$, we now show that these four expressions are equal to $0$; it will follow in these cases that $w_0 \equiv \sigma((a b \overline{a} \overline{b})^n)$ for some permutation $\sigma$, and hence $[\phi_2(w_0)] = [\phi_3(w_0)]$.

If $\pi$ maps $x \mapsto x, y \mapsto y$ or $x \mapsto \overline{x}, y \mapsto \overline{y}$, consider $(\overline{y}x\overline{y})_{\phi_1(w_0)} = (\overline{y}x\overline{y})_{\pi \phi_4(w_0)}$.
This is equivalent to
\[
	(\overline{y}xx\overline{y})_{w_0} = (\overline{y}x\overline{y})_{w_0} + (\overline{y}xx)_{w_0} + (xx\overline{y})_{w_0} + (xxx)_{w_0},
\]
which simplifies to $0 = (y\overline{x}y)_{w_0}$ since $w_0$ is alternating.

If $\pi$ maps $x \mapsto \overline{x}, y \mapsto y$ or $x \mapsto x, y \mapsto \overline{y}$, consider $(yxy)_{\phi_1(w_0)} = (yxy)_{\pi \phi_4(w_0)} = (y\overline{x}y)_{\phi_4(w_0)}$.
Therefore $0 = (y\overline{x}y)_{w_0}$.

If $\pi$ maps $x \mapsto y, y \mapsto x$ or $x \mapsto \overline{y}, y \mapsto \overline{x}$, then by Lemma~\ref{permutation} we have $\phi_1(w_0) \equiv \pi \phi_4(w_0) = (\{\pi(x)\}, \pi(\overline{y})) \pi(w_0) \equiv \phi_2 \pi(w_0)$.
As in the proof of Lemma~\ref{s1 = s3}, $\phi_1^2(w_0) \equiv \pi(w_0)$ implies $(yxy)_{w_0} = 0$.

It remains to address the two order-$4$ permutations mapping $x \mapsto \overline{y}, y \mapsto x$ and $x \mapsto y, y \mapsto \overline{x}$.
Let $\pi$ be either of these permutations.
By Lemma~\ref{permutation}, $\phi_1(w_0) \equiv \pi \phi_4(w_0) \equiv \phi_1 \pi(w_0)$.
Hence $w_0 \equiv \pi(w_0)$.
Since the conclusion holds for the empty word, assume $w_0$ is nonempty.
Let $k \geq 1$ be minimal such that $w_0 = \rho^k \sigma(w_0)$ for some $\sigma \in \{\pi, \pi^{-1}\}$, where again $\rho$ is rotation to the right by one character.
Let $u$ be the prefix of $w_0$ of length $k$.
Then $w_0 = u \cdot \sigma(u) \cdot \sigma^2(u) \cdot \sigma^3(u) \cdots \sigma^{-1}(u)$.
Since $\sigma$ has order $4$, we have $w_0 = (u \cdot \sigma(u) \cdot \sigma^2(u) \cdot \sigma^3(u))^n = \left(\prod_{i=0}^3 \sigma^i(u) \right)^n$, where $n = \frac{|w_0|}{4k}$.
Therefore $\sigma(w_0) \equiv w_0$.
By Lemma~\ref{permutation} and Equation~\eqref{inner automorphism},
\begin{align*}
	\sigma \phi_2(w_0)
	&= (\{\sigma(y)\}, \sigma(\overline{x})) \sigma(w_0) \\
	&\equiv \phi_3 \sigma(w_0) \\
	&\equiv \phi_3(w_0),
\end{align*}
so $[\phi_2(w_0)] = [\phi_3(w_0)]$.
\end{proof}

Experimental evidence suggests that in fact the previous three lemmas can be generalized, but we do not have a proof.

\begin{conjecture*}
Lemmas~\ref{s1 = s2}, \ref{s1 = s3}, and \ref{s1 = s4} remain true if we remove the requirement that $w_0$ is alternating.
\end{conjecture*}

For example, $aaaa$ satisfies the condition $[\phi_1(w_0)] = [\phi_2(w_0)]$ of Lemma~\ref{s1 = s2} and also the conclusion $[\phi_3(w_0)] = [\phi_4(w_0)]$.
Examples for Lemmas~\ref{s1 = s3} and \ref{s1 = s4} are, respectively, $aabb$ and $aabb\overline{a}\overline{a}\overline{b}\overline{b}$.

To classify the graphs of root classes containing an alternating minimal word, however, we only need the lemmas as stated.

\begin{proof}[Proof of Theorem~\ref{alternating root class}]
First we establish the uniqueness of an alternating word vertex $[w_0]$ in $\Gamma(W)$ and that every other vertex is a neighbor of $[w_0]$.
Let $w_0 \in W$ be an alternating root word.
If $[w_0]$ is the only vertex of $\Gamma(W)$, then it is clearly the unique vertex containing alternating minimal words.
Otherwise, let $x, y \in L_2$ such that $[\phi_1(w_0)] \neq [w_0]$.
By Lemma~\ref{alternating fixed}, $\phi_1(w_0)$ is not alternating.
Thus, by Lemma~\ref{level images}, the outdegree of $[\phi_1(w_0)]$ is $2$.
As in the proof of Lemma~\ref{s1 = s2}, the principal automorphisms that are level on $\phi_1(w_0)$ are $\phi_2$ and $\phi_4$.
The image of $\phi_1(w_0)$ under $\phi_2$ is $w_0$, and by Corollary~\ref{triangle corollary} the vertex $[\phi_4 \phi_1(w_0)]$ is connected to $[w_0]$ by a one-letter automorphism.
That is, any edge from $[\phi_1(w_0)]$ that does not point to $[w_0]$ points to a neighbor of $[w_0]$ (possibly to $[\phi_1(w_0)]$ itself).
Since $\Gamma(W)$ is connected, this implies that every vertex other than $[w_0]$ is, in fact, a neighbor of $[w_0]$.
Lemma~\ref{alternating fixed} now implies that $[w_0]$ is the unique vertex in $\Gamma(W)$ containing an alternating minimal word.

By Lemma~\ref{alternating fixed}, if $[w_0]$ is connected to itself by $[\phi]$ for some one-letter automorphism $\phi$ then $[w_0]$ is also connected to itself by $[\phi^{-1}]$.
In other words, loops on $[w_0]$ come in pairs of inverse automorphisms.
We consider separately the cases that $[w_0]$ has $4$, $2$, or $0$ loops.

If $[w_0]$ has $4$ loops, then $\Gamma(W)$ is (R4), a single vertex with four loops.

Suppose $[w_0]$ has exactly $2$ loops.
Let $x, y$ be such that $[\phi_1(w_0)] = [w_0] = [\phi_2(w_0)]$.
By Lemma~\ref{s1 = s2}, $[\phi_3(w_0)] = [\phi_4(w_0)]$, so $\Gamma(W)$ has exactly two vertices, $[w_0]$ and $[\phi_3(w_0)]$.
Since $\phi_3$ and $\phi_4$ are inequivalent modulo $\Inn F_2$, two edges connect $[w_0]$ to $[\phi_3(w_0)]$.
This accounts for all four edges emanating from $[w_0]$, so it suffices to determine the edges from $[\phi_3(w_0)]$.
The one-letter automorphisms that are level on $\phi_3(w_0)$ are $\phi_4$ and $\phi_2$.
Moreover, $\phi_4 \phi_3(w_0) = w_0$ and by Corollary~\ref{triangle corollary} $[\phi_2 \phi_3(w_0)] = [\phi_1(w_0)] = [w_0]$.
There are therefore two edges from $[\phi_3(w_0)]$ to $[w_0]$, so $\Gamma(W)$ is (R5).

Finally, suppose that $[w_0]$ has no loops.
Corollary~\ref{triangle corollary} implies that $[\phi_1(w_0)]$ is connected to $[\phi_3(w_0)]$ by a one-letter automorphism and that $[\phi_2(w_0)]$ is connected to $[\phi_4(w_0)]$ by a one-letter automorphism (allowing the possibility that these edges may be loops).
If $[w_0]$ has four distinct neighbors, then, since $[w_0]$ is the only vertex with outdegree $4$, the outdegree of each other vertex is $2$, and it follows that $\Gamma(W)$ is the bow tie (R7).
If $[w_0]$ has fewer than four neighbors, then there is at least one pair of identified images of $w_0$.
The $\binom{4}{2} = 6$ possibilities are as follows.

If $[\phi_1(w_0)] = [\phi_2(w_0)]$, then $[\phi_3(w_0)] = [\phi_4(w_0)]$ by Lemma~\ref{s1 = s2}.
Therefore $[w_0]$ has exactly two neighbors, each of which has outdegree $2$.
Moreover, two edges connect $[w_0]$ to each of its neighbors.
Therefore $\Gamma(W)$ is (R6).

If $[\phi_1(w_0)] = [\phi_3(w_0)]$, then the proof of Lemma~\ref{s1 = s3} shows that $\phi_1(w_0)$ is alternating.
Therefore $\phi_1(w_0) = w_0$ by Lemma~\ref{alternating fixed}, contradicting our assumption that $w_0$ has no loops.

If $[\phi_1(w_0)] = [\phi_4(w_0)]$, then $[\phi_2(w_0)] = [\phi_3(w_0)]$ by Lemma~\ref{s1 = s4}.
The vertices of $\Gamma(W)$ are as in the case $[\phi_1(w_0)] = [\phi_2(w_0)]$, with analogous edges, so $\Gamma(W)$ is (R6).

The remaining three cases are equivalent under permutations to the first three.

If $[\phi_3(w_0)] = [\phi_4(w_0)]$, then let $\sigma$ be the permutation that maps $x \mapsto y, y \mapsto x$.
Then $[\phi_1 \sigma(w_0)] = [\phi_2 \sigma(w_0)]$, which is the first case we considered, so $\Gamma(W)$ is (R6).

If $[\phi_2(w_0)] = [\phi_4(w_0)]$, letting $\sigma$ map $x \mapsto \overline{x}, y \mapsto y$ gives $[\phi_1 \sigma(w_0)] = [\phi_3 \sigma(w_0)]$, which is the second case and so does not occur when $[w_0]$ has no loops.

If $[\phi_3(w_0)] = [\phi_2(w_0)]$, then $[\phi_1 \sigma(w_0)] = [\phi_4 \sigma(w_0)]$, where $\sigma$ maps $x \mapsto y, y \mapsto x$.
This is the third case, so $\Gamma(W)$ is (R6).
\end{proof}

\section{Enumeration}\label{Enumeration}

Having classified automorphic conjugacy classes of $F_2$ in this paper, it is natural to ask how many automorphic conjugacy classes contain minimal words of length $n$.
In this section we make some observations that suggest the intriguing possibility of an exact enumeration.
We restrict our speculation to non-root classes, which outnumber root classes (at least for $5 \leq n \leq 20$ and probably for $n > 20$ as well).

In Section~\ref{Non-root classes} we mentioned that for $9 \leq n \leq 20$ there are precisely $5$ (P1) classes of size $|V(\Gamma(W))| = n-5$ (the largest possible size, per Theorem~\ref{automorphic conjugacy class size}).
This can be clearly seen in Appendix~\ref{Number of paths of each size} as an eventually constant diagonal of $5$s in the table enumerating (P1) classes.
Our first conjecture is that all diagonals of this table are eventually constant.
The tables enumerating (P2) and (P3) classes, which result from folding, suggest that these classes have size at most $n/2$ for $n \geq 2$, so we phrase the conjecture as follows.

\begin{conjecture*}
Fix $k \geq 0$.
The number of automorphic conjugacy classes of $F_2$ of size $n - k$ whose minimal words have length $n$ is constant for sufficiently large $n$.
\end{conjecture*}

For $k = 0, 1, 2, \dots$, these constants appear to be
\begin{equation}\label{limiting counts}
	0, 0, 0, 0, 0, 5, 12, 17, 24, 67, 196, 437, \dots.
\end{equation}
A simple expression for the $k$th term of this sequence is not obvious.
However, refining our parameterization of classes reveals additional structure.

Define the \emph{weight} of a word $w$ to be $\min((a)_w, (b)_w)$.
Suppose $\phi = (\{y\}, x)$ is level on a minimal word $w$ of length $n$.
Then $(y)_w = (y)_{\phi(w)}$, and hence $(x)_w = n - (y)_w = n - (y)_{\phi(w)} = (x)_{\phi(w)}$.
Therefore the weight of a minimal word is preserved under level one-letter automorphisms.
The weight is also preserved under inner automorphisms and permutations, so the weight is invariant on all minimal words in an automorphic conjugacy class $W$.

Let us count classes not by size alone but by size and weight.
There is only one class of weight $0$ for each $n \geq 0$, namely the class containing $a^n$, which has size $1$.
There are no classes of weight $1$, since $a^{n-1}b$ is not minimal.

We return to (P1) classes.
For $9 \leq n \leq 20$, the $5$ classes of type (P1) and size $n-5$ all have weight $4$.
Similarly, for $10 \leq n \leq 20$, all $12$ second-largest (P1) classes (those of size $n-6$) have weight $4$.
The $17$ third-largest classes all have weight $4$, and the $24$ fourth-largest classes also all have weight $4$.
However, not all $67$ fifth-largest classes have weight $4$; it turns out that $29$ have weight $4$ and $38$ have weight $6$.
If, instead of Sequence~\eqref{limiting counts}, we consider the number of classes (for sufficiently large $n$) of size $n - k$ whose minimal words have length $n$ and weight $4$, we obtain the sequence
\[
	0, 0, 0, 0, 0, 5, 12, 17, 24, 29, 36, 41, \dots,
\]
whose terms are given by a simple expression.
Namely, this sequence is eventually a linear quasi-polynomial with modulus $2$.

\begin{conjecture*}
For $k \geq 4$ and $n \geq \max(2k-2,9)$, the number of (P1) classes of size $n - k$ whose minimal words have length $n$ and weight $4$ is
\[
	\begin{cases}
		6 k - 24	& \text{if $k \equiv 0 \mod 2$} \\
		6 k - 25	& \text{if $k \equiv 1 \mod 2$}.
	\end{cases}
\]
\end{conjecture*}

It appears that all classes of odd weight have size $1$.
For even weights, however, we see behavior similar to weight-$4$ classes.
For example, fixing $k$, the number of (P1) classes of size $n-k$ and weight $6$ appears to be constant for $n \geq 2k-5$, with values $38, 160, 396, 800$ for $k = 9, \dots, 12$.
These four terms are not enough to guess a reliable expression for the $k$th term, but we suspect it is given by a quasi-polynomial as well.

Therefore it seems that sufficiently large (P1) classes should be amenable to enumeration.
Analogous conjectures for (P2) and (P3) classes aren't quite as strongly suggested by the data available in Appendix~\ref{Number of paths of each size}, but we are still willing to state the following.

\begin{conjecture*}
Fix an odd $k \geq 1$.
The number of (P2) classes of size $(n - k)/2$ whose minimal words have length $n$ is constant for sufficiently large odd $n$.
\end{conjecture*}

\begin{conjecture*}
Fix an even $k \geq 0$.
The number of (P3) classes of size $(n - k)/2$ whose minimal words have length $n$ is constant for sufficiently large even $n$.
\end{conjecture*}

On the other side of the spectrum, counting small classes as opposed to large classes seems promising as well.
Let us consider classes of size $1$, which for $0 \leq n \leq 20$ account for more than half of all classes whose minimal words have length $n$ (nearly $88\%$ for $n = 20$).
For odd weights, the number of size-$1$ classes appears to be given by a polynomial.

\begin{conjecture*}
For $n \geq 7$, the number of non-root classes of size $1$ whose minimal words have length $n$ and weight $3$ is $3 n - 11$.
\end{conjecture*}

\begin{conjecture*}
For $n \geq 11$, the number of non-root classes of size $1$ whose minimal words have length $n$ and weight $5$ is
\[
	\frac{1}{6} \left(35 n^3 - 645 n^2 + 3988 n - 8262\right).
\]
\end{conjecture*}

For even weights, the expressions seem to be quasi-polynomials rather than polynomials.

\begin{conjecture*}
For $n \geq 5$, the number of non-root classes of size $1$ whose minimal words have length $n$ and weight $2$ is
\[
	\begin{cases}
		n - 2	& \text{if $n \equiv 0 \mod 2$} \\
		n - 3	& \text{if $n \equiv 1 \mod 2$}.
	\end{cases}
\]
\end{conjecture*}

\begin{conjecture*}
For $n \geq 9$, the number of non-root classes of size $1$ whose minimal words have length $n$ and weight $4$ is
\[
	\begin{cases}
		\left(2 n^3 - 36 n^2 + 244 n - 540\right)/6	& \text{if $n \equiv 0 \mod 4$} \\
		\left(2 n^3 - 36 n^2 + 241 n - 537\right)/6	& \text{if $n \equiv 1 \mod 4$} \\
		\left(2 n^3 - 36 n^2 + 244 n - 546\right)/6	& \text{if $n \equiv 2 \mod 4$} \\
		\left(2 n^3 - 36 n^2 + 241 n - 537\right)/6	& \text{if $n \equiv 3 \mod 4$}.
	\end{cases}
\]
\end{conjecture*}

We leave these conjectures and their generalizations as open problems.
The referee has pointed out that, aside from independent interest, knowing the number of automorphic conjugacy classes of a given size would allow one to compute the expected size $|V(\Gamma(W))|$ of a random class $W$ whose minimal words have length $n$.
There are sufficiently many classes of size $1$ that for each $0 \leq n \leq 20$ this number lies in the interval $[1, 1.76)$, with the value for $n = 20$ being approximately $1.18$.
Does the expected size of a random class lie in the interval $[1, 2)$ for all $n \geq 0$?
Does the expected size of a random class tend to $1$ as $n$ gets large?

\section*{Acknowledgement}

We thank the referee for several good suggestions.

\newpage

\appendix

\section{Table of automorphic conjugacy classes}\label{Table of automorphic conjugacy classes}

The following tables list all automorphic conjugacy classes containing a word of length $n \leq 9$.
For a given length, classes are sorted first by size and then by the lexicographically least word.
Representatives modulo $J$ of minimal words in each class are given, and each class is identified by its graph type in Theorems~\ref{non-root class}, \ref{non-alternating root class}, and \ref{alternating root class}.
Data files listing all automorphic conjugacy classes containing a word of length $n \leq 20$ can be downloaded from the second author's web site\footnote{\url{http://thales.math.uqam.ca/~rowland/data/automorphic_conjugacy_classes.html} as of this writing.}.

\

\[
\begin{array}{|rl|rl|rl|} \cline{1-2}\cline{3-4}\cline{5-6}
\text{0.1 (R4)} & \epsilon & \text{7.7 (P1)} & aaab\overline{a}bb & \text{8.20 (P1)} & aaabb\overline{a}\overline{a}b \\\cline{1-2}\cline{3-4}\cline{5-6}
\text{1.1 (R1)} & a & \text{7.8 (P1)} & aaab\overline{a}\overline{b}\overline{b} & \text{8.21 (P1)} & aaabb\overline{a}\overline{a}\overline{b} \\\cline{1-2}\cline{3-4}\cline{5-6}
\text{2.1 (P3)} & aa & \text{7.9 (P1)} & aaabba\overline{b} & \text{8.22 (P1)} & aaabba\overline{b}\overline{b} \\\cline{1-2}\cline{3-4}\cline{5-6}
\text{3.1 (P3)} & aaa & \text{7.10 (P1)} & aaabb\overline{a}b & \text{8.23 (P1)} & aaabb\overline{a}bb \\\cline{1-2}\cline{3-4}\cline{5-6}
\text{4.1 (P3)} & aaaa & \text{7.11 (P1)} & aaabb\overline{a}\overline{b} & \text{8.24 (P1)} & aaabb\overline{a}\overline{b}\overline{b} \\\cline{1-2}\cline{3-4}\cline{5-6}
\text{4.2 (R4)} & ab\overline{a}\overline{b} & \text{7.12 (P3)} & aaab\overline{a}\overline{a}\overline{b} & \text{8.25 (P1)} & aaabbba\overline{b} \\\cline{1-2}\cline{3-4}\cline{5-6}
\text{4.3 (R5)} & aabb & \text{7.13 (P1)} & aabaa\overline{b}\overline{b} & \text{8.26 (P1)} & aaabbb\overline{a}\overline{b} \\\cline{3-4}\cline{5-6}
		& aba\overline{b} & \text{7.14 (P1)} & aabb\overline{a}\overline{a}\overline{b} & \text{8.27 (P3)} & aaab\overline{a}\overline{a}\overline{a}\overline{b} \\\cline{1-2}\cline{3-4}\cline{5-6}
\text{5.1 (P3)} & aaaaa & \text{7.15 (P1)} & aabba\overline{b}\overline{b} & \text{8.28 (P1)} & aabbaa\overline{b}\overline{b} \\\cline{1-2}\cline{3-4}\cline{5-6}
\text{5.2 (P3)} & aaba\overline{b} & \text{7.16 (P2)} & aaaaabb & \text{8.29 (P1)} & aabb\overline{a}\overline{a}\overline{b}\overline{b} \\\cline{1-2}\cline{5-6}
\text{5.3 (P3)} & aab\overline{a}\overline{b} & 		& aaaab\overline{a}b & \text{8.30 (R1)} & aabb\overline{a}\overline{b}\overline{a}\overline{b} \\\cline{1-2}\cline{5-6}
\text{5.4 (P2)} & aaabb & 		& aaab\overline{a}\overline{a}b & \text{8.31 (R4)} & ab\overline{a}\overline{b}ab\overline{a}\overline{b} \\\cline{3-4}\cline{5-6}
		& aab\overline{a}b & \text{8.1 (P3)} & aaaaaaaa & \text{8.32 (R2)} & aabab\overline{a}bb \\\cline{1-2}\cline{3-4}
\text{6.1 (P3)} & aaaaaa & \text{8.2 (P3)} & aaaaaba\overline{b} & 		& aab\overline{a}bba\overline{b} \\\cline{1-2}\cline{3-4}\cline{5-6}
\text{6.2 (P3)} & aaaba\overline{b} & \text{8.3 (P1)} & aaaaabbb & \text{8.33 (R2)} & aaba\overline{b}a\overline{b}\overline{b} \\\cline{1-2}\cline{3-4}
\text{6.3 (P1)} & aaabbb & \text{8.4 (P3)} & aaaaab\overline{a}\overline{b} & 		& aaba\overline{b}\overline{b}\overline{a}\overline{b} \\\cline{1-2}\cline{3-4}\cline{5-6}
\text{6.4 (P3)} & aaab\overline{a}\overline{b} & \text{8.5 (P3)} & aaaabaa\overline{b} & \text{8.34 (R5)} & aabbaabb \\\cline{1-2}\cline{3-4}
\text{6.5 (P3)} & aabaa\overline{b} & \text{8.6 (P1)} & aaaaba\overline{b}\overline{b} & 		& aba\overline{b}aba\overline{b} \\\cline{1-2}\cline{3-4}\cline{5-6}
\text{6.6 (P1)} & aaba\overline{b}\overline{b} & \text{8.7 (P1)} & aaaab\overline{a}bb & \text{8.35 (R5)} & aabba\overline{b}\overline{a}b \\\cline{1-2}\cline{3-4}
\text{6.7 (P1)} & aabba\overline{b} & \text{8.8 (P1)} & aaaab\overline{a}\overline{b}\overline{b} & 		& aba\overline{b}ab\overline{a}\overline{b} \\\cline{1-2}\cline{3-4}\cline{5-6}
\text{6.8 (P1)} & aabb\overline{a}\overline{b} & \text{8.9 (P1)} & aaaabba\overline{b} & \text{8.36 (R5)} & aab\overline{a}\overline{b}\overline{b}\overline{a}b \\\cline{1-2}\cline{3-4}
\text{6.9 (P3)} & aab\overline{a}\overline{a}\overline{b} & \text{8.10 (P1)} & aaaabb\overline{a}b & 		& aba\overline{b}\overline{a}b\overline{a}\overline{b} \\\cline{1-2}\cline{3-4}\cline{5-6}
\text{6.10 (P3)} & aaaabb & \text{8.11 (P1)} & aaaabb\overline{a}\overline{b} & \text{8.37 (R3)} & aaababbb \\\cline{3-4}
		& aaab\overline{a}b & \text{8.12 (P1)} & aaaabbbb & 		& aababa\overline{b}\overline{b} \\\cline{3-4}
		& aab\overline{a}\overline{a}b & \text{8.13 (P3)} & aaaab\overline{a}\overline{a}\overline{b} & 		& aabba\overline{b}a\overline{b} \\\cline{1-2}\cline{3-4}\cline{5-6}
\text{7.1 (P3)} & aaaaaaa & \text{8.14 (P3)} & aaabaaa\overline{b} & \text{8.38 (R6)} & aaabbabb \\\cline{1-2}\cline{3-4}
\text{7.2 (P3)} & aaaaba\overline{b} & \text{8.15 (P1)} & aaabaa\overline{b}\overline{b} & 		& aaba\overline{b}\overline{b}ab \\\cline{1-2}\cline{3-4}
\text{7.3 (P1)} & aaaabbb & \text{8.16 (P1)} & aaab\overline{a}\overline{a}bb & 		& ababa\overline{b}a\overline{b} \\\cline{1-2}\cline{3-4}\cline{5-6}
\text{7.4 (P3)} & aaaab\overline{a}\overline{b} & \text{8.17 (P1)} & aaab\overline{a}\overline{a}\overline{b}\overline{b} & \text{8.39 (R3)} & aababba\overline{b} \\\cline{1-2}\cline{3-4}
\text{7.5 (P3)} & aaabaa\overline{b} & \text{8.18 (P1)} & aaaba\overline{b}\overline{b}\overline{b} & 		& aababb\overline{a}b \\\cline{1-2}\cline{3-4}
\text{7.6 (P1)} & aaaba\overline{b}\overline{b} & \text{8.19 (P1)} & aaabbaa\overline{b} & 		& aaba\overline{b}\overline{a}\overline{b}\overline{b} \\\hline
\end{array}
\]

\newpage

\[
\begin{array}{|rl|rl|rl|} \cline{1-2}\cline{3-4}\cline{5-6}
\text{8.40 (R3)} & aababb\overline{a}\overline{b} & \text{9.17 (P1)} & aaaab\overline{a}\overline{a}\overline{b}\overline{b} & \text{9.49 (P3)} & aab\overline{a}\overline{b}aba\overline{b} \\\cline{3-4}\cline{5-6}
		& aabab\overline{a}\overline{b}\overline{b} & \text{9.18 (P1)} & aaaaba\overline{b}\overline{b}\overline{b} & \text{9.50 (P3)} & aab\overline{a}\overline{b}ab\overline{a}\overline{b} \\\cline{3-4}\cline{5-6}
		& aabb\overline{a}\overline{b}a\overline{b} & \text{9.19 (P1)} & aaaab\overline{a}bbb & \text{9.51 (P3)} & aab\overline{a}\overline{b}\overline{a}ba\overline{b} \\\cline{1-2}\cline{3-4}\cline{5-6}
\text{8.41 (P3)} & aaaaaabb & \text{9.20 (P1)} & aaaab\overline{a}\overline{b}\overline{b}\overline{b} & \text{9.52 (P3)} & aab\overline{a}\overline{b}\overline{a}b\overline{a}\overline{b} \\\cline{3-4}\cline{5-6}
		& aaaaab\overline{a}b & \text{9.21 (P1)} & aaaabbaa\overline{b} & \text{9.53 (P1)} & aaaababbb \\\cline{3-4}
		& aaaab\overline{a}\overline{a}b & \text{9.22 (P1)} & aaaabb\overline{a}\overline{a}b & 		& aaabb\overline{a}b\overline{a}b \\\cline{3-4}\cline{5-6}
		& aaab\overline{a}\overline{a}\overline{a}b & \text{9.23 (P1)} & aaaabb\overline{a}\overline{a}\overline{b} & \text{9.54 (P1)} & aaaabbabb \\\cline{1-2}\cline{3-4}
\text{8.42 (R7)} & aaba\overline{b}abb & \text{9.24 (P1)} & aaaabba\overline{b}\overline{b} & 		& aaab\overline{a}bb\overline{a}b \\\cline{3-4}\cline{5-6}
		& aabba\overline{b}ab & \text{9.25 (P1)} & aaaabb\overline{a}bb & \text{9.55 (P1)} & aaaabbbab \\\cline{3-4}
		& aabba\overline{b}\overline{a}\overline{b} & \text{9.26 (P1)} & aaaabb\overline{a}\overline{b}\overline{b} & 		& aaab\overline{a}b\overline{a}bb \\\cline{3-4}\cline{5-6}
		& aabb\overline{a}\overline{b}\overline{a}b & \text{9.27 (P1)} & aaaabbba\overline{b} & \text{9.56 (P1)} & aaababa\overline{b}\overline{b} \\\cline{3-4}
		& abab\overline{a}ba\overline{b} & \text{9.28 (P1)} & aaaabbb\overline{a}b & 		& aaabba\overline{b}a\overline{b} \\\cline{1-2}\cline{3-4}\cline{5-6}
\text{8.43 (R7)} & aaba\overline{b}\overline{a}bb & \text{9.29 (P1)} & aaaabbb\overline{a}\overline{b} & \text{9.57 (P1)} & aaababba\overline{b} \\\cline{3-4}
		& aaba\overline{b}\overline{b}\overline{a}b & \text{9.30 (P3)} & aaaab\overline{a}\overline{a}\overline{a}\overline{b} & 		& aaabb\overline{a}ba\overline{b} \\\cline{3-4}\cline{5-6}
		& aabb\overline{a}\overline{b}ab & \text{9.31 (P1)} & aaabaaa\overline{b}\overline{b} & \text{9.58 (P1)} & aaababb\overline{a}b \\\cline{3-4}
		& aab\overline{a}bb\overline{a}\overline{b} & \text{9.32 (P1)} & aaabaa\overline{b}\overline{b}\overline{b} & 		& aababba\overline{b}\overline{b} \\\cline{3-4}\cline{5-6}
		& abab\overline{a}b\overline{a}\overline{b} & \text{9.33 (P1)} & aaabb\overline{a}\overline{a}\overline{a}b & \text{9.59 (P1)} & aaababb\overline{a}\overline{b} \\\cline{1-2}\cline{3-4}
\text{9.1 (P3)} & aaaaaaaaa & \text{9.34 (P1)} & aaabb\overline{a}\overline{a}\overline{a}\overline{b} & 		& aaabb\overline{a}b\overline{a}\overline{b} \\\cline{1-2}\cline{3-4}\cline{5-6}
\text{9.2 (P3)} & aaaaaaba\overline{b} & \text{9.35 (P1)} & aaabbaa\overline{b}\overline{b} & \text{9.60 (P1)} & aaabab\overline{a}bb \\\cline{1-2}\cline{3-4}
\text{9.3 (P1)} & aaaaaabbb & \text{9.36 (P1)} & aaabb\overline{a}\overline{a}bb & 		& aabab\overline{a}\overline{a}bb \\\cline{1-2}\cline{3-4}\cline{5-6}
\text{9.4 (P3)} & aaaaaab\overline{a}\overline{b} & \text{9.37 (P1)} & aaabb\overline{a}\overline{a}\overline{b}\overline{b} & \text{9.61 (P1)} & aaabab\overline{a}\overline{b}\overline{b} \\\cline{1-2}\cline{3-4}
\text{9.5 (P3)} & aaaaabaa\overline{b} & \text{9.38 (P1)} & aaabba\overline{b}\overline{b}\overline{b} & 		& aaabb\overline{a}\overline{b}a\overline{b} \\\cline{1-2}\cline{3-4}\cline{5-6}
\text{9.6 (P1)} & aaaaaba\overline{b}\overline{b} & \text{9.39 (P1)} & aaabbbaa\overline{b} & \text{9.62 (P1)} & aaaba\overline{b}a\overline{b}\overline{b} \\\cline{1-2}\cline{3-4}
\text{9.7 (P1)} & aaaaab\overline{a}bb & \text{9.40 (P1)} & aaabbb\overline{a}\overline{a}\overline{b} & 		& aaaba\overline{b}\overline{b}\overline{a}\overline{b} \\\cline{1-2}\cline{3-4}\cline{5-6}
\text{9.8 (P1)} & aaaaab\overline{a}\overline{b}\overline{b} & \text{9.41 (P1)} & aaabbba\overline{b}\overline{b} & \text{9.63 (P1)} & aaaba\overline{b}\overline{a}\overline{b}\overline{b} \\\cline{1-2}\cline{3-4}
\text{9.9 (P1)} & aaaaabba\overline{b} & \text{9.42 (P1)} & aaabbb\overline{a}\overline{b}\overline{b} & 		& aaaba\overline{b}\overline{b}a\overline{b} \\\cline{1-2}\cline{3-4}\cline{5-6}
\text{9.10 (P1)} & aaaaabb\overline{a}b & \text{9.43 (P2)} & aababaa\overline{b}\overline{b} & \text{9.64 (P2)} & aaaba\overline{b}\overline{b}ab \\\cline{1-2}\cline{3-4}
\text{9.11 (P1)} & aaaaabb\overline{a}\overline{b} & \text{9.44 (P3)} & aaba\overline{b}aba\overline{b} & 		& aaba\overline{b}a\overline{b}ab \\\cline{1-2}\cline{3-4}\cline{5-6}
\text{9.12 (P1)} & aaaaabbbb & \text{9.45 (P3)} & aaba\overline{b}ab\overline{a}\overline{b} & \text{9.65 (P1)} & aaabbaba\overline{b} \\\cline{1-2}\cline{3-4}
\text{9.13 (P3)} & aaaaab\overline{a}\overline{a}\overline{b} & \text{9.46 (P3)} & aaba\overline{b}\overline{a}ba\overline{b} & 		& aaab\overline{a}bba\overline{b} \\\cline{1-2}\cline{3-4}\cline{5-6}
\text{9.14 (P3)} & aaaabaaa\overline{b} & \text{9.47 (P3)} & aaba\overline{b}\overline{a}b\overline{a}\overline{b} & \text{9.66 (P1)} & aaabbab\overline{a}b \\\cline{1-2}\cline{3-4}
\text{9.15 (P1)} & aaaabaa\overline{b}\overline{b} & \text{9.48 (P2)} & aabbaa\overline{b}\overline{a}\overline{b} & 		& aabba\overline{b}\overline{b}ab \\\cline{1-2}
\text{9.16 (P1)} & aaaab\overline{a}\overline{a}bb &  & & & \\\hline
\end{array}
\]

\newpage

\[
\begin{array}{|rl|rl|rl|} \cline{1-2}\cline{3-4}\cline{5-6}
\text{9.67 (P1)} & aaabbab\overline{a}\overline{b} & \text{9.82 (P1)} & aaabbaabb & \text{9.93 (P1)} & aabbaab\overline{a}\overline{b} \\
		& aaab\overline{a}bb\overline{a}\overline{b} & 		& aabbaab\overline{a}b & 		& aabba\overline{b}\overline{a}\overline{a}b \\\cline{1-2}
\text{9.68 (P1)} & aaabba\overline{b}\overline{a}\overline{b} & 		& aab\overline{a}bab\overline{a}b & 		& aab\overline{a}bab\overline{a}\overline{b} \\\cline{3-4}\cline{5-6}
		& aaab\overline{a}ba\overline{b}\overline{b} & \text{9.83 (P1)} & aaabbabbb & \text{9.94 (P1)} & aabbaa\overline{b}\overline{a}b \\\cline{1-2}
\text{9.69 (P1)} & aaabb\overline{a}bab & 		& aabaabb\overline{a}b & 		& aabba\overline{b}\overline{b}\overline{a}b \\
		& aabba\overline{b}a\overline{b}\overline{b} & 		& aab\overline{a}b\overline{a}bab & 		& aab\overline{a}\overline{b}a\overline{b}\overline{a}\overline{b} \\\cline{1-2}\cline{3-4}\cline{5-6}
\text{9.70 (P1)} & aaabb\overline{a}\overline{b}\overline{a}\overline{b} & \text{9.84 (P1)} & aabaaba\overline{b}\overline{b} & \text{9.95 (P1)} & aabba\overline{b}\overline{a}\overline{a}\overline{b} \\
		& aaab\overline{a}b\overline{a}\overline{b}\overline{b} & 		& aabaa\overline{b}\overline{b}ab & 		& aabb\overline{a}\overline{b}\overline{b}\overline{a}b \\\cline{1-2}
\text{9.71 (P1)} & aaab\overline{a}babb & 		& aababa\overline{b}a\overline{b} & 		& aab\overline{a}ba\overline{b}\overline{a}\overline{b} \\\cline{3-4}\cline{5-6}
		& aaba\overline{b}\overline{b}abb & \text{9.85 (P1)} & aabaabba\overline{b} & \text{9.96 (P2)} & aaaaaaabb \\\cline{1-2}
\text{9.72 (P1)} & aaab\overline{a}bbab & 		& aabab\overline{a}ba\overline{b} & 		& aaaaaab\overline{a}b \\
		& aababb\overline{a}\overline{a}b & 		& aabba\overline{b}\overline{a}\overline{b}\overline{b} & 		& aaaaab\overline{a}\overline{a}b \\\cline{1-2}\cline{3-4}
\text{9.73 (P1)} & aaab\overline{a}\overline{b}a\overline{b}\overline{b} & \text{9.86 (P1)} & aabaabb\overline{a}\overline{b} & 		& aaaab\overline{a}\overline{a}\overline{a}b \\\cline{5-6}
		& aaab\overline{a}\overline{b}\overline{b}\overline{a}\overline{b} & 		& aabab\overline{a}b\overline{a}\overline{b} & \text{9.97 (P1)} & aaaba\overline{b}abb \\\cline{1-2}
\text{9.74 (P1)} & aaab\overline{a}\overline{b}\overline{a}\overline{b}\overline{b} & 		& aaba\overline{b}\overline{a}\overline{a}bb & 		& aaabb\overline{a}\overline{b}\overline{a}b \\\cline{3-4}
		& aaab\overline{a}\overline{b}\overline{b}a\overline{b} & \text{9.87 (P1)} & aabaab\overline{a}\overline{b}\overline{b} & 		& aaba\overline{b}ab\overline{a}b \\\cline{1-2}
\text{9.75 (P2)} & aaab\overline{a}\overline{b}\overline{b}\overline{a}b & 		& aabaa\overline{b}\overline{b}\overline{a}b & 		& aab\overline{a}b\overline{a}\overline{b}\overline{a}b \\\cline{5-6}
		& aab\overline{a}\overline{b}a\overline{b}\overline{a}b & 		& aabab\overline{a}\overline{b}a\overline{b} & \text{9.98 (P1)} & aaaba\overline{b}\overline{a}bb \\\cline{1-2}\cline{3-4}
\text{9.76 (P1)} & aabaa\overline{b}a\overline{b}\overline{b} & \text{9.88 (P1)} & aabaa\overline{b}abb & 		& aaabb\overline{a}\overline{b}ab \\
		& aabaa\overline{b}\overline{b}\overline{a}\overline{b} & 		& aaba\overline{b}\overline{a}\overline{a}\overline{b}\overline{b} & 		& aaba\overline{b}\overline{a}b\overline{a}b \\\cline{1-2}
\text{9.77 (P1)} & aabaa\overline{b}\overline{a}\overline{b}\overline{b} & 		& aaba\overline{b}\overline{a}\overline{b}a\overline{b} & 		& aab\overline{a}b\overline{a}\overline{b}ab \\\cline{3-4}\cline{5-6}
		& aabaa\overline{b}\overline{b}a\overline{b} & \text{9.89 (P1)} & aabaa\overline{b}\overline{a}bb & \text{9.99 (P1)} & aaaba\overline{b}\overline{b}\overline{a}b \\\cline{1-2}
\text{9.78 (P1)} & aababb\overline{a}\overline{a}\overline{b} & 		& aabb\overline{a}\overline{a}\overline{b}ab & 		& aaab\overline{a}\overline{b}\overline{b}ab \\
		& aababb\overline{a}\overline{b}\overline{b} & 		& aab\overline{a}\overline{b}\overline{a}\overline{b}a\overline{b} & 		& aaba\overline{b}a\overline{b}\overline{a}b \\\cline{1-2}\cline{3-4}
\text{9.79 (P1)} & aabab\overline{a}\overline{a}\overline{b}\overline{b} & \text{9.90 (P1)} & aaba\overline{b}aabb & 		& aab\overline{a}\overline{b}a\overline{b}ab \\\cline{5-6}
		& aabb\overline{a}\overline{a}\overline{b}a\overline{b} & 		& aabb\overline{a}\overline{b}\overline{a}\overline{a}b & \text{9.100 (P1)} & aaabba\overline{b}ab \\\cline{1-2}
\text{9.80 (P1)} & aaba\overline{b}\overline{b}\overline{a}bb & 		& aab\overline{a}baba\overline{b} & 		& aaab\overline{a}\overline{b}\overline{a}bb \\\cline{3-4}
		& aabb\overline{a}\overline{b}\overline{b}ab & \text{9.91 (P1)} & aaba\overline{b}aa\overline{b}\overline{b} & 		& aab\overline{a}ba\overline{b}ab \\\cline{1-2}
\text{9.81 (P1)} & aaabaabbb & 		& aaba\overline{b}a\overline{b}\overline{a}\overline{b} & 		& aab\overline{a}\overline{b}\overline{a}b\overline{a}b \\\cline{5-6}
		& aabaab\overline{a}bb & 		& aaba\overline{b}\overline{b}\overline{a}\overline{a}\overline{b} & \text{9.101 (P1)} & aaabba\overline{b}\overline{a}b \\\cline{3-4}
		& aabab\overline{a}b\overline{a}b & \text{9.92 (P1)} & aaba\overline{b}\overline{b}\overline{a}\overline{a}b & 		& aaab\overline{a}\overline{b}abb \\
 & & 		& aabb\overline{a}\overline{b}\overline{a}\overline{a}\overline{b} & 		& aab\overline{a}ba\overline{b}\overline{a}b \\
 & & 		& aab\overline{a}b\overline{a}\overline{b}\overline{a}\overline{b} & 		& aab\overline{a}\overline{b}ab\overline{a}b \\\hline
\end{array}
\]

\newpage

\section{Number of automorphic conjugacy classes of each type}\label{Number of automorphic conjugacy classes of each type}

This table gives the number of automorphic conjugacy classes whose minimal words have length $n$ for each graph type in Theorems~\ref{non-root class}, \ref{non-alternating root class}, and \ref{alternating root class}.

\

\[
\begin{array}{r|cccccccccc}
n & \text{(P1)} & \text{(P2)} & \text{(P3)} & \text{(R1)} & \text{(R2)} & \text{(R3)} & \text{(R4)} & \text{(R5)} & \text{(R6)} & \text{(R7)} \\ \hline
0 & \gray{0} & \gray{0} & \gray{0} & \gray{0} & \gray{0} & \gray{0} & 1 & \gray{0} & \gray{0} & \gray{0} \\
1 & \gray{0} & \gray{0} & 1 & \gray{0} & \gray{0} & \gray{0} & \gray{0} & \gray{0} & \gray{0} & \gray{0} \\
2 & \gray{0} & \gray{0} & 1 & \gray{0} & \gray{0} & \gray{0} & \gray{0} & \gray{0} & \gray{0} & \gray{0} \\
3 & \gray{0} & \gray{0} & 1 & \gray{0} & \gray{0} & \gray{0} & \gray{0} & \gray{0} & \gray{0} & \gray{0} \\
4 & \gray{0} & \gray{0} & 1 & \gray{0} & \gray{0} & \gray{0} & 1 & 1 & \gray{0} & \gray{0} \\
5 & \gray{0} & 1 & 3 & \gray{0} & \gray{0} & \gray{0} & \gray{0} & \gray{0} & \gray{0} & \gray{0} \\
6 & 4 & \gray{0} & 6 & \gray{0} & \gray{0} & \gray{0} & \gray{0} & \gray{0} & \gray{0} & \gray{0} \\
7 & 10 & 1 & 5 & \gray{0} & \gray{0} & \gray{0} & \gray{0} & \gray{0} & \gray{0} & \gray{0} \\
8 & 22 & \gray{0} & 8 & 1 & 2 & 3 & 1 & 3 & 1 & 2 \\
9 & 81 & 5 & 15 & \gray{0} & \gray{0} & \gray{0} & \gray{0} & \gray{0} & \gray{0} & \gray{0} \\
10 & 298 & 4 & 38 & \gray{0} & \gray{0} & \gray{0} & \gray{0} & \gray{0} & \gray{0} & \gray{0} \\
11 & 855 & 7 & 49 & \gray{0} & \gray{0} & \gray{0} & \gray{0} & \gray{0} & \gray{0} & \gray{0} \\
12 & 2140 & 4 & 96 & 4 & 12 & 244 & 1 & 7 & 5 & 31 \\
13 & 7040 & 29 & 155 & \gray{0} & \gray{0} & \gray{0} & \gray{0} & \gray{0} & \gray{0} & \gray{0} \\
14 & 22244 & 30 & 342 & \gray{0} & \gray{0} & \gray{0} & \gray{0} & \gray{0} & \gray{0} & \gray{0} \\
15 & 64774 & 49 & 553 & \gray{0} & \gray{0} & \gray{0} & \gray{0} & \gray{0} & \gray{0} & \gray{0} \\
16 & 175209 & 46 & 1104 & 11 & 70 & 10899 & 1 & 19 & 15 & 380 \\
17 & 543631 & 185 & 1927 & \gray{0} & \gray{0} & \gray{0} & \gray{0} & \gray{0} & \gray{0} & \gray{0} \\
18 & 1649842 & 232 & 3892 & \gray{0} & \gray{0} & \gray{0} & \gray{0} & \gray{0} & \gray{0} & \gray{0} \\
19 & 4824825 & 343 & 6889 & \gray{0} & \gray{0} & \gray{0} & \gray{0} & \gray{0} & \gray{0} & \gray{0} \\
20 & 13535352 & 406 & 13592 & 35 & 400 & 473355 & 1 & 55 & 51 & 4547
\end{array}
\]

\newpage

\section{Number of paths of each size}\label{Number of paths of each size}

The following table gives the number of (P1) classes $W$ whose minimal words have length $n$ and whose graph $\Gamma(W)$ has $m$ vertices.
Zeros are omitted.

\tiny
\[
\begin{array}{r|ccccccccccccccc}
n & m = 1 & 2 & 3 & 4 & 5 & 6 & 7 & 8 & 9 & 10 & 11 & 12 & 13 & 14 & 15 \\ \hline
0 \\
1 \\
2 \\
3 \\
4 \\
5 \\
6 & 4 \\
7 & 10 \\
8 & 22 \\
9 & 35 & 26 & 15 & 5 \\
10 & 224 & 35 & 22 & 12 & 5 \\
11 & 741 & 44 & 33 & 20 & 12 & 5 \\
12 & 1984 & 53 & 40 & 29 & 17 & 12 & 5 \\
13 & 4538 & 1964 & 401 & 76 & 27 & 17 & 12 & 5 \\
14 & 17064 & 3762 & 1052 & 236 & 72 & 24 & 17 & 12 & 5 \\
15 & 55096 & 6433 & 2279 & 633 & 205 & 70 & 24 & 17 & 12 & 5 \\
16 & 158613 & 10156 & 4197 & 1440 & 477 & 201 & 67 & 24 & 17 & 12 & 5 \\
17 & 415072 & 110789 & 12916 & 3041 & 1043 & 446 & 199 & 67 & 24 & 17 & 12 & 5 \\
18 & 1353447 & 250705 & 35075 & 6714 & 2250 & 888 & 442 & 196 & 67 & 24 & 17 & 12 & 5 \\
19 & 4197308 & 513440 & 89404 & 16198 & 4995 & 1862 & 857 & 440 & 196 & 67 & 24 & 17 & 12 & 5 \\
20 & 12303132 & 968489 & 204968 & 40097 & 11122 & 4226 & 1707 & 853 & 437 & 196 & 67 & 24 & 17 & 12 & 5
\end{array}
\]
\normalsize

\vspace{1cm}

The following tables give the number of (P2) (left table) and (P3) (right table) classes $W$ whose minimal words have length $n$ and whose graph $\Gamma(W)$ has $m$ vertices.

\tiny
\[
\begin{array}{r|ccccccccc}
n & m = 1 & 2 & 3 & 4 & 5 & 6 & 7 & 8 & 9 \\ \hline
0 \\
1 \\
2 \\
3 \\
4 \\
5 & \gray{0} & 1 \\
6 \\
7 & \gray{0} & \gray{0} & 1 \\
8 \\
9 & 2 & 2 & \gray{0} & 1 \\
10 & 2 & 2 \\
11 & 2 & 2 & 2 & \gray{0} & 1 \\
12 & 2 & 2 \\
13 & 18 & 6 & 2 & 2 & \gray{0} & 1 \\
14 & 22 & 6 & 2 \\
15 & 26 & 12 & 6 & 2 & 2 & \gray{0} & 1 \\
16 & 30 & 14 & 2 \\
17 & 138 & 26 & 10 & 6 & 2 & 2 & \gray{0} & 1 \\
18 & 188 & 36 & 6 & 2 \\
19 & 242 & 58 & 22 & 10 & 6 & 2 & 2 & \gray{0} & 1 \\
20 & 308 & 82 & 14 & 2
\end{array}
\qquad
\begin{array}{r|cccccccccc}
n & m = 1 & 2 & 3 & 4 & 5 & 6 & 7 & 8 & 9 & 10 \\ \hline
0 \\
1 & 1 \\
2 & 1 \\
3 & 1 \\
4 & 1 \\
5 & 3 \\
6 & 5 & \gray{0} & 1 \\
7 & 5 \\
8 & 7 & \gray{0} & \gray{0} & 1 \\
9 & 15 \\
10 & 31 & 4 & 2 & \gray{0} & 1 \\
11 & 49 \\
12 & 85 & 4 & 4 & 2 & \gray{0} & 1 \\
13 & 155 \\
14 & 301 & 28 & 8 & 2 & 2 & \gray{0} & 1 \\
15 & 553 \\
16 & 1031 & 44 & 16 & 8 & 2 & 2 & \gray{0} & 1 \\
17 & 1927 \\
18 & 3659 & 172 & 38 & 12 & 6 & 2 & 2 & \gray{0} & 1 \\
19 & 6889 \\
20 & 13123 & 336 & 82 & 28 & 12 & 6 & 2 & 2 & \gray{0} & 1
\end{array}
\]

\end{document}